\documentclass[12pt]{amsart}
\usepackage{amscd,amsmath,amssymb,amsfonts}
\usepackage[cmtip, all]{xy}

\newtheorem{thm}{Theorem}[section]
\newtheorem{prop}[thm]{Proposition}
\newtheorem{lem}[thm]{Lemma}
\newtheorem{cor}[thm]{Corollary}

\theoremstyle{definition}

\newtheorem{defn}[thm]{Definition}
\newtheorem{remk}[thm]{Remark}
\newtheorem{remks}[thm]{Remarks}

\newtheorem{exm}[thm]{Example}
\newtheorem{exms}[thm]{Examples}
\newtheorem{notat}[thm]{Notation}

\numberwithin{equation}{section}

{\hfill$\square$\end{defn}}

{\hfill$\square$\end{remk}}

{\hfill$\square$\end{remks}}

{\hfill$\square$\end{exm}}

{\hfill$\square$\end{exms}}

{\hfill$\square$\end{notat}}

\newcommand{\sC}{{\mathcal C}}

\newcommand{\sR}{{\mathcal R}}

\newcommand{\sV}{{\mathcal V}}

\newcommand{\sZ}{{\mathcal Z}}
\newcommand{\A}{{\mathbb A}}

\newcommand{\G}{{\mathbb G}}

\newcommand{\bL}{{\mathbb L}}

\newcommand{\N}{{\mathbb N}}
\renewcommand{\P}{{\mathbb P}}
\newcommand{\Q}{{\mathbb Q}}
\newcommand{\R}{{\mathbb R}}

\newcommand{\Z}{{\mathbb Z}}

\newcommand{\surj}{\twoheadrightarrow}
\newcommand{\inj}{\hookrightarrow}

\newcommand{\Pic}{{\rm Pic}}

\newcommand{\Hom}{{\rm Hom}}

\newcommand{\0}{\emptyset}

\newcommand{\ds}{{/\kern-3pt/}}

\newcommand{\ov}{\overline}
\newcommand{\wt}{\widetilde}

\newcommand{\tuborg}{\left\{\begin{array}{ll}}
\newcommand{\sluttuborg}{\end{array}\right.}

\begin{document}
\title{Cobordism ring of toric varieties}
\author{Amalendu Krishna, V. Uma}
\address{School of Mathematics, Tata Institute of Fundamental Research,  
Homi Bhabha Road, Colaba, Mumbai, India.} 
\email{amal@math.tifr.res.in}
\address{Department of Mathematics, IIT Madras, Chennai, India}
\email{vuma@iitm.ac.in}

\baselineskip=10pt 
  
\keywords{Algebraic cobordism, toric varieties}        

\subjclass[2010]{Primary 14C25; Secondary 19E15}
\maketitle
\begin{abstract}
We describe the equivariant algebraic cobordism rings of
smooth toric varieties. This equivariant description is used 
to compute the ordinary cobordism ring of such varieties. 
\end{abstract}

\section{Introduction}
Let $k$ be a field of characteristic zero. A scheme (or variety) in this
text will mean a quasi-projective $k$-scheme. Based on the construction of
the motivic algebraic cobordism spectrum $MGL$ by Voevodsky \cite{Voev},
Levine and Morel \cite{LM} invented a geometric version of the algebraic 
cobordism theory some years ago. They showed that the resulting cohomology
theory $\Omega^*(-)$ is the universal oriented cohomology theory on the category
of smooth varieties over $k$. Moreover, it is a universal Borel-Moore
oriented homology theory on the category of all $k$-schemes. It was later
shown by Levine \cite{Levine2} that there is isomorphism of cohomology
theories $\Omega^*(-) \xrightarrow{\cong} MGL^{2*,*}(-)$. In other words,
the algebraic cobordism theory of Levine and Morel computes pieces of the
motivic cobordism theory of Voevodsky. Since the cobordism theory of Levine
and Morel has products, one knows that for a smooth scheme $X$, the
cohomology $\Omega^*(X)$ is a graded ring.  

Since the algebraic version of the cobordism theory was discovered only 
recently, there are not many computations of this cohomology theory
known at this stage. Levine and Morel proved a projective bundle formula
for the algebraic cobordism, from which one can deduce the formula for the
cobordism ring of projective spaces. In a recent work \cite{HK}, Hornbostel 
and Kiritchencko computed the cobordism ring of the complete flag variety
${GL_n}/B$ using the techniques similar to the known Schubert calculus
for the singular cohomology. These results were also obtained by
Calm{\`e}s, Petrov and Zainoulline in \cite{CPZ}. The rational cobordism rings of 
arbitrary flag varieties and flag bundles were recently described by the 
first author in \cite{Krishna3} and \cite{Krishna1}. One of the
principal goals of this paper is to describe the cobordism ring of smooth
toric varieties. 

Our strategy of computing the ordinary cobordism ring of smooth toric varieties
is to use the powerful technique of the equivariant cohomology. Such techniques
have been very effective in computing the ordinary Chow ring of varieties
with group action. We refer the reader to the seminal paper \cite{Brion}  
for many results in this direction. The equivariant algebraic cobordism 
groups for smooth varieties were initially defined by Deshpande \cite{DD}. 
They were subsequently developed into a complete theory of equivariant 
cobordism for all $k$-schemes in \cite{Krishna4}. This theory is based on the 
analogous construction of the equivariant Chow groups by Totaro \cite{Totaro1}
and Edidin-Graham \cite{EG}. All the expected  basic properties
of the equivariant cobordism theory were established in \cite{Krishna4}.
These properties are similar to the analogous properties of the 
ordinary (non-equivariant) cobordism which were established by Levine and Morel 
\cite{LM}.  

In the first main result of this paper, we exploit the work of 
\cite{Krishna4, Krishna3} to compute the equivariant algebraic cobordism ring 
of smooth toric varieties. Our description of the equivariant cobordism ring
is very similar to the one for the equivariant Chow ring of toric varieties
in \cite{Brion}. In order to apply the results of \cite{Krishna4, Krishna3}, 
we first prove a decomposition theorem 
({\sl cf.} Theorem~\ref{thm:decomposition*}) for the equivariant cobordism
ring of a smooth toric variety. Such a decomposition of the equivariant
cobordism ring is obtained by adapting the techniques of Vezzosi and
Vistoli \cite{VV} who invented these techniques in the context of the
equivariant algebraic $K$-theory of smooth varieties with torus action.
We deduce the formula for the ordinary cobordism ring of a smooth toric
variety from the equivariant cobordism ring using \cite[Theorem~3.4]{Krishna3},
which gives an explicit description of the ordinary cobordism ring of a 
variety with a torus action as a quotient of the equivariant cobordism ring.

To state our main results, let $T$ be a split torus of rank $n$ over $k$ and 
let $M$ denote the lattice of the one-parameter subgroups of $T$. We identify
$M^{\vee}$ with group of characters of $T$. Let $<,> : M \times M^{\vee} \to \Z$
be the natural pairing. Let 
$X = X(\Delta)$ be a smooth toric variety associated to a fan 
$\Delta$ in $M_{\R}$. We refer the reader to \cite{Fulton1} for the basics 
of toric varieties. Let $\Delta_1$ denote the set of one-dimensional cones
in $\Delta$. Let $\{v_{\rho}| \rho \in \Delta_1\}$ denote the set of 
primitive vectors in $M$ along the one-dimensional faces of $\Delta$. 
Let $\Delta^0_1$ denote the collection of subsets $S \subseteq \Delta_1$ such 
that $S = \{\rho_1, \cdots , \rho_s\}$ is not contained in any maximal cone of 
$\Delta$. 

Let $S = S(T)$ denote the
$T$-equivariant cobordism ring $\Omega^*_T(k)$ of the base field.
Let $\Omega^*_T(X)$ and $\Omega^*(X)$ denote the $T$-equivariant
and the ordinary cobordism rings of $X$ respectively.
For $\rho \in \Delta_1$, the corresponding orbit closure $V_{\rho}$ is a
$T$-invariant smooth closed subvariety of $X$ which is a Weil
divisor on $X$. Let $[V_{\rho}]$ denote the cobordism cycle $[V_{\rho} \to X]$
in $\Omega^*(X)$. 

Let $\bL$ denote the Lazard ring. It is known that $\bL = 
{\underset{i \le 0}\oplus}\bL_i$ is a graded ring which is canonically
isomorphic to $\Omega^*(k)$. There is a formal (commutative) group law
on $\bL$ represented by a power series $F(u,v)$ in $\bL[[u,v]]$ such that 
$(\bL, F)$ is the universal commutative formal group law of rank one.  

Let $\bL[[t_{\rho}]]$ denote the graded power 
series ring over $\bL$ in the variables $\{t_{\rho}| \rho \in \Delta_1\}$.
For $x_1, \cdots, x_s \in \bL[[t_{\rho}]]$ and $n_i \in \Z$, let
$\stackrel{s}{\underset{i = 1}\sum} [n_i]_F x_i$ denote the sum
according to the formal group law $F$ ({\sl cf.} Section~\ref{section:AC}). 
Let $I_{\Delta}$ denote the graded ideal of $\bL[[t_{\rho}]]$ generated by the 
set of monomials 
$\{{\underset{\rho \in S}\prod} t_{\rho} | S \in  \Delta^0_1\}$.  
Let $\ov{I}_{\Delta}$ denote the graded ideal of the polynomial ring 
$\bL[t_{\rho}]$ generated by the set of monomials 
$\{{\underset{\rho \in S}\prod} t_{\rho} | S \in  \Delta^0_1\} \bigcup
\{t^{n+1}_{\rho}| \rho \in \Delta_1\}$.
The following results are the main features of this paper.
\begin{thm}\label{thm:SRP-Main**}
For a smooth toric variety $X = X(\Delta)$ associated to a fan $\Delta$ in
$M_{\R}$, there is a natural map of $S$-algebras
\[
\Psi_X : \frac{\bL[[t_{\rho}]]}{I_{\Delta}} \to \Omega^*_T(X)
\]
which is an isomorphism.
\end{thm} 

\begin{thm}\label{thm:COBT**}
Let $X = X(\Delta)$ be a smooth toric variety associated to a fan $\Delta$
in $M_{\R}$. Then the assignment $t_{\rho} \mapsto [V_{\rho}]$ defines an
$\bL$-algebra isomorphism
\begin{equation}\label{eqn:COBT1}
\ov{\Psi}_X: \frac{\bL[t_{\rho}]}
{\left(\ov{I}_{\Delta}, \underset{\rho \in \Delta_1}\sum
[<\chi, v_{\rho}>]_F \ t_{\rho}\right)} \to \Omega^*(X),
\end{equation}
where $\chi$ runs over ${M}^{\vee}$.
\end{thm}

Since $\Omega^*(-)$ is the universal oriented cohomology theory, we see at
once that all other oriented cohomology theories (e.g. $K$-theory,
Chow groups, \'etale and singular cohomology) of smooth toric varieties
can be deduced from Theorem~\ref{thm:COBT**} in a simple way. 
 
As one would also expect, the computation of any oriented cohomology ring of
toric varieties is the foundational step in computing an oriented
cohomology of various spherical varieties. It is known ({\sl cf.}
\cite[Theorem~8.7]{Krishna4}) that the rational equivariant cobordism of
a variety $X$ with the action of a reductive group is the subgroup of
invariants under the Weyl group action on the equivariant cobordism of $X$ for 
the action of a maximal torus of $G$. This reduces most of the computations
of the cobordism ring of spherical varieties to the case of the
cobordism ring of toric varieties. The results of this paper will be used
in \cite{HK} to compute the cobordism ring of certain spherical varieties.
If $X$ is a complex toric variety, then the techniques of this paper 
and the known relation between the algebraic and the complex cobordism can
be easily used to compute the complex cobordism of any smooth toric variety.
The complex cobordism ring of toric manifolds in topology were described
earlier by Buchstaber and Ray in \cite{BR}.   

We end the introduction with a brief outline of this paper.
We recollect all the definitions and the relevant basic properties
of the ordinary and the equivariant cobordism in the next section.
In Section~\ref{section:PRELIMS}, we establish some preliminary results
which are frequently used in this paper. We also introduce the notion
of cohomological rigidity for the equivariant cobordism in this section
and describe its consequences. In Sections~\ref{section:STRATA} and
~\ref{section:DECOMP}, our goal is to prove our main intermediate
result which describes the equivariant cobordism ring as a subring
of the equivariant cobordism rings of the orbits corresponding to the maximal
cones in the associated fan. This is the crucial ingredient in the proof
of our main result about the equivariant cobordism ring of smooth toric
varieties. In Section~\ref{section:SRP}, we prove 
Theorem~\ref{thm:SRP-Main} which describes the equivariant cobordism ring
in terms of a Stanley-Reisner algebra. The ordinary cobordism ring of
a smooth toric variety is obtained in the final section using
Theorem~\ref{thm:SRP-Main} and \cite[Theorem~3.4]{Krishna3}.

\section{Recollection of equivariant cobordism}\label{section:AC}
Let $k$ be a field of characteristic zero.
Since we shall be concerned with the study of 
schemes with group actions and the associated quotient schemes, and since 
such quotients often require the original scheme to be quasi-projective,
we shall assume throughout this paper that all schemes over $k$ are 
quasi-projective. 

In this section, we briefly recall the definition of equivariant
algebraic cobordism and some of its main properties from \cite{Krishna4}.
Since most of the results of \cite{Krishna4} will be repeatedly used in this 
text, we summarize them here for reader's convenience. 
For the definition and all details about the algebraic
cobordism used in this paper, we refer the reader to the work of Levine and
Morel \cite{LM}. 
\\

{\bf Notations.} We shall denote the category of quasi-projective $k$-schemes 
by $\sV_k$. By a scheme, we shall mean an object 
of $\sV_k$. The category of smooth quasi-projective schemes
will be denoted by $\sV^S_k$. If $G$ is a linear algebraic group over $k$, we 
shall denote the category of quasi-projective $k$-schemes with a $G$-action
and $G$-equivariant maps by $\sV_G$. The associated category of smooth
$G$-schemes will be denoted by $\sV_G^S$. All $G$-actions in this paper will be
assumed to be linear. Recall that this means that all $G$-schemes are
assumed to admit $G$-equivariant ample line bundles. This assumption is
always satisfied for normal schemes ({\sl cf.} \cite[Theorem~2.5]{Sumihiro}, 
\cite[5.7]{Thomason3}). In particular, any action of $G$ on a toric
variety is linear.
\\

Recall that the Lazard ring $\bL$ is a polynomial ring over $\Z$ on infinite 
but countably many variables and is given by the quotient of the polynomial 
ring $\Z[A_{ij}| (i,j) \in \N^2]$ by the relations, which uniquely define the 
universal formal group law $F_{\bL}$ of rank one on $\bL$. 
Recall that a cobordism cycle over a $k$-scheme $X$ is
a family $\alpha = [Y \xrightarrow{f} X, L_1, \cdots , L_r]$,
where $Y$ is a smooth scheme, the map $f$ is projective, and $L_i$'s are line 
bundles on $Y$. Here, one allows the set of line bundles to be empty.
The degree of such a cobordism cycle is  defined to be ${\rm deg}(\alpha) =
{\rm dim}_k(Y)-r$ and its codimension is defined to be ${\rm dim}(X) -
{\rm deg}(\alpha)$. If $\sZ_*(X)$ is the free abelian group generated by the 
cobordism cycles of the above type with $Y$ irreducible, then
$\sZ_*(X)$ is graded by the degree of cycles.
The algebraic cobordism group of $X$ is defined as
\[
\Omega_*(X) = \frac{\sZ_*(X)}{\sR_*(X)},
\]
where ${\sR_*(X)}$ is the graded subgroup generated by relations which are
determined by the dimension and the section axioms and the above formal group 
law. If $X$ is equi-dimensional, we set $\Omega^i(X) = \Omega_{{\rm dim}(X)-i}(X)$
and grade $\Omega^*(X)$ by the codimension of the cobordism cycles. 
It was shown by Levine and Pandharipande \cite{LP} that the cobordism group
$\Omega_*(X)$ can also be defined as the quotient
\begin{equation}\label{eqn:LP*}
\Omega_*(X) = \frac{\sZ'_*(X)}{\sR'_*(X)},
\end{equation}
where $\sZ'_*(X)$ is the free abelian group on cobordism cycles 
$[Y \xrightarrow{f} X]$ with $Y$ smooth and irreducible and $f$ projective.
The graded subgroup $\sR'_*(X)$ is generated by cycles satisfying the relation
of {\sl double point degeneration}.

Let $X$ be a $k$-scheme of dimension $d$. For $j \in \Z$, let $Z_j$ be the
set of all closed subschemes $Z \subset X$ such that ${\rm dim}_k(Z) \le j$
(we assume ${\rm dim}(\0) = - \infty$). The set $Z_j$ is then ordered by 
the inclusion. For $i \ge 0$,  we set
\[
\Omega_i(Z_j) = {\underset{Z \in Z_j} \varinjlim} \Omega_i(Z) 
\ \ {\rm and} \ \
\Omega_*(Z_j) = {\underset{i \ge 0} \bigoplus} \ \Omega_i(Z_j).
\]
It is immediate that $\Omega_*(Z_j)$ is a graded $\bL$-module and there is
a graded $\bL$-linear map $\Omega_*(Z_j) \to \Omega_*(X)$.
We define $F_j\Omega_*(X)$ to be the image of the natural $\bL$-linear map
$\Omega_*(Z_j) \to \Omega_*(X)$. In other words, $F_j\Omega_*(X)$ is the image
of all $\Omega_*(W) \to \Omega_*(X)$, where $W \to X$ is a projective map
such that ${\rm dim}({\rm Image}(W)) \le j$.
One checks at once that there is a canonical {\sl niveau filtration}
\begin{equation}\label{eqn:niveau1}
0 = F_{-1}\Omega_*(X) \subseteq F_0\Omega_*(X) \subseteq \cdots \subseteq
F_{d-1}\Omega_*(X) \subseteq F_d\Omega_*(X) = \Omega_*(X).
\end{equation}

\subsection{Equivariant cobordism}
In this text, $G$ will denote a linear algebraic group of dimension $g$ 
over $k$. This group will be a torus for the most parts of this paper.
All representations of $G$ will be finite dimensional. 
Recall form \cite{Krishna4} that for any integer $j \ge 0$, a {\sl good pair}  
$\left(V_j,U_j\right)$ corresponding to $j$ for the $G$-action is a pair 
consisting of a $G$-representation $V_j$ and an open subset $U_j \subset V_j$ 
such that the codimension of the complement is at least $j$ and $G$ acts 
freely on $U_j$ with quotient ${U_j}/G$ a quasi-projective scheme.
It is known that such good pairs always exist. 
  
Let $X$ be a $k$-scheme of dimension $d$ with a $G$-action.
For $j \ge 0$, let $(V_j, U_j)$ be an $l$-dimensional good pair 
corresponding to $j$. For $i \in \Z$, if we set
\begin{equation}\label{eqn:E-cob*}
{\Omega^G_i(X)}_j =  \frac{\Omega_{i+l-g}\left({X\stackrel{G} {\times} U_j}\right)}
{F_{d+l-g-j}\Omega_{i+l-g}\left({X\stackrel{G} {\times} U_j}\right)},
\end{equation}
then it is known that ${\Omega^G_i(X)}_j$ is independent of the 
choice of the good pair $(V_j, U_j)$. 
It is also known that for each $j \ge 0$, ${\Omega^G_*(X)}_j =
{\underset{i \in \Z}\bigoplus} \ {\Omega^G_i(X)}_j$ is a graded $\bL$-module. 
Moreover, there is a natural surjective 
map $\Omega^G_*(X)_{j'} \surj \Omega^G_*(X)_j$ of graded $\bL$-modules for 
$j' \ge j \ge 0$.

\begin{defn}\label{defn:ECob}
Let $X$ be a $k$-scheme of dimension $d$ with a $G$-action. For any 
$i \in \Z$, we define the {\sl equivariant algebraic cobordism} of $X$ to be 
\[
\Omega^G_i(X) = {\underset {j} \varprojlim} \ \Omega^G_i(X)_j.
\]
\end{defn}
The reader should note from the above definition that unlike the ordinary
cobordism, the equivariant algebraic cobordism $\Omega^G_i(X)$ can be 
non-zero for any $i \in \Z$. We set 
\[
\Omega^G_*(X) = {\underset{i \in \Z} \bigoplus} \ \Omega^G_i(X).
\]
If $X$ is an equi-dimensional $k$-scheme with $G$-action, we let
$\Omega^i_G(X) = \Omega^G_{d-i}(X)$ and $\Omega^*_G(X) =
{\underset{i \in \Z} \oplus} \ \Omega^i_G(X)$. 
It is known that if $G$ is trivial, then the $G$-equivariant cobordism
reduces to the ordinary one.
\begin{remk}\label{remk:Csmooth}
It is easy to check from the above definition of the niveau filtration 
that if $X$ is a smooth and irreducible $k$-scheme of dimension $d$,
then $F_j\Omega_i(X) = F^{d-j}\Omega^{d-i}(X)$, where $F^{\bullet}\Omega^*(X)$ 
is the coniveau filtration used in \cite{DD}. Furthermore, one also 
checks in this case that if $G$ acts on $X$, then
\begin{equation}\label{eqn:Csmooth1}
\Omega^i_G(X) = {\underset {j} \varprojlim} \
\frac{\Omega^{i}\left({X\stackrel{G} {\times} U_j}\right)}
{F^{j}\Omega^{i}\left({X\stackrel{G} {\times} U_j}\right)},
\end{equation}
where $(V_j, U_j)$ is a good pair corresponding to any $j \ge 0$.
Thus the above definition of the equivariant cobordism
coincides with that of \cite{DD} for smooth schemes.
\end{remk}

The following important result shows that if 
we suitably choose a sequence of good pairs ${\{(V_j, U_j)\}}_{j \ge 0}$, then 
the above equivariant cobordism group can be computed without taking 
quotients by the niveau filtration. This is often very helpful in
computing the equivariant cobordism groups. It is moreover known \cite{EG} that 
such a sequence of good pairs always exist.
  
\begin{thm}\label{thm:NO-Niveu}$(${\sl cf.} \cite[Theorem~6.1]{Krishna4}$)$
Let ${\{(V_j, U_j)\}}_{j \ge 0}$ be a sequence of $l_j$-dimensional 
good pairs such that \\
$(i)$ $V_{j+1} = V_j \oplus W_j$ as representations of $G$ with ${\rm dim}(W_j) 
> 0$ and \\
$(ii)$ $U_j \oplus W_j \subset U_{j+1}$ as $G$-invariant open subsets. \\ 
Then for any scheme $X$ as above and any $i \in \Z$,
\[
\Omega^G_i(X) \xrightarrow{\cong} {\underset{j}\varprojlim} \
\Omega_{i+l_j-g}\left(X \stackrel{G}{\times} U_j\right).
\]
\end{thm}

For equi-dimensional schemes, we shall write the (equivariant) cobordism groups
cohomologically.
The $G$-equivariant cobordism group $\Omega^*(k)$ of the ground field $k$ 
is denoted by $\Omega^*(BG)$ and is called the cobordism of the 
{\sl classifying space} of $G$. We shall often write it as $S(G)$.

\subsection{Change of groups}
If $H \subset G$ is a closed subgroup of dimension $h$, then any 
$l$-dimensional good pair 
$(V_j, U_j)$ for $G$-action is also a good pair for the induced $H$-action. 
Moreover, for any $X \in \sV_G$ of dimension $d$, 
$X \stackrel{H}{\times} U_j \to X \stackrel{G}{\times} U_j$
is an \'etale locally trivial $G/H$-fibration and hence a smooth map 
({\sl cf.} \cite[Theorem~6.8]{Borel}) of relative dimension $g-h$. 
This induces the pull-back map
\begin{equation}\label{eqn:res}
r^G_{H,X} : \Omega^G_*(X) \to \Omega^H_*(X).
\end{equation}
Taking $H = \{1\}$, we get the {\sl forgetful} map 
\begin{equation}\label{eqn:res}
r^G_X : \Omega^G_*(X) \to \Omega_*(X)
\end{equation}
from the equivariant to the non-equivariant (ordinary) cobordism groups. 
Since $r^G_{H,X}$ is obtained as a pull-back under the smooth map, it commutes
with any projective push-forward and smooth pull-back ({\sl cf.} 
Theorem~\ref{thm:Basic}). 

The equivariant cobordism for the action of a group $G$ is related with
the equivariant cobordism for the action of the various subgroups of $G$
by the following. We refer to [{\sl loc. cit.}, Proposition~5.5] for a proof.

\begin{prop}[Morita Isomorphism]\label{prop:Morita}
Let $H \subset G$ be a closed subgroup and let $X \in {\sV}_H$.
Then there is a canonical isomorphism
\begin{equation}\label{eqn:MoritaI}
\Omega^G_*\left(G \stackrel{H}{\times} X\right) \xrightarrow{\cong}
\Omega^H_*(X).
\end{equation}
\end{prop}

\subsection{Fundamental class of cobordism cycles}
\label{subsection:FundC}
Let $X \in \sV_G$ and let $Y \xrightarrow{f} X$ be a morphism in $\sV_G$
such that $Y$ is smooth of dimension $d$ and $f$ is projective. For any
$j \ge 0$ and any $l$-dimensional good pair $(V_j, U_j)$, 
$[Y_G \xrightarrow{f_G} X_G]$ is an ordinary cobordism cycle of dimension
$d+l-g$ by [{\sl loc.cit.}, Lemma~5.3] and hence defines an element 
$\alpha_j \in {\Omega^G_d(X)}_j$. Moreover, it is evident that the image of 
$\alpha_{j'}$ is $\alpha_j$ for $j' \ge j$. Hence we get a unique element
$\alpha \in \Omega^G_d(X)$, called the {\sl G-equivariant fundamental class}
of the cobordism cycle $[Y \xrightarrow{f} X]$. We also see from this
more generally that if $[Y \xrightarrow{f} X, L_1, \cdots, L_r]$ is as
above with each $L_i$ a $G$-equivariant line bundle on $Y$, then this defines
a unique class in $\Omega^G_{d-r}(X)$. It is interesting question to ask under 
what conditions on the group $G$, the equivariant cobordism group 
$\Omega^G_*(X)$ is generated by the fundamental classes of $G$-equivariant
cobordism cycles on $X$ as $S(G)$-module. This is known to be true for torus
actions by \cite[Theorem~4.11]{Krishna3}.

\subsection{Basic properties}
The following result summarizes the basic properties of the equivariant
cobordism.

\begin{thm}\label{thm:Basic}$(${\sl cf.} \cite[Theorems~5.1, 5.4]{Krishna4}$)$
The equivariant algebraic cobordism satisfies the following properties. \\
$(i)$ {\sl Functoriality :} The assignment $X \mapsto \Omega_*(X)$ is
covariant for projective maps and contravariant for smooth maps in
$\sV_G$. It is also contravariant for l.c.i. morphisms in $\sV_G$. 
Moreover, for a fiber diagram 
\[
\xymatrix@C.7pc{
X' \ar[r]^{g'} \ar[d]_{f'} & X \ar[d]^{f} \\
Y' \ar[r]_{g} & Y}
\]
in $\sV_G$ with $f$ projective and $g$ smooth, one has 
$g^* \circ f_* = {f'}_* \circ {g'}^* : \Omega^G_*(X) \to \Omega^G_*(Y')$.
\\
$(ii) \ Localization :$ For a $G$-scheme $X$ and a closed $G$-invariant
subscheme $Z \subset X$ with complement $U$, 
there is an exact sequence
\[
\Omega^G_*(Z) \to \Omega^G_*(X) \to \Omega^G_*(U) \to 0.
\]
$(iii) \ Homotopy :$  If $f : E \to X$ is a $G$-equivariant vector bundle,
then $f^*: \Omega^G_*(X) \xrightarrow{\cong} \Omega^G_*(E)$. \\
$(iv) \ Chern \ classes :$ For any $G$-equivariant vector bundle $E
\xrightarrow{f} X$ of rank $r$, there are equivariant Chern class operators
$c^G_m(E) : \Omega^G_*(X) \to \Omega^G_{*-m}(X)$ for $0 \le m \le r$ with
$c^G_0(E) = 1$. These Chern classes 
have same functoriality properties as in the non-equivariant case. Moreover, 
they satisfy the Whitney sum formula. \\
$(v) \ Free \ action :$ If $G$ acts freely on $X$ with quotient $Y$, then
$\Omega^G_*(X) \xrightarrow{\cong} \Omega_*(Y)$. \\
$(vi) \ Exterior \ Product :$ There is a natural product map
\[
\Omega^G_i(X) \otimes_{\Z} \Omega^G_{i'}(X') \to \Omega^G_{i+i'}(X \times X').
\]
In particular, $\Omega^G_*(k)$ is a graded $\bL$-algebra and $\Omega^G_*(X)$ is
a graded $\Omega^G_*(k)$-module for every $X \in \sV_G$. 
For $X$ smooth, the pull-back via the diagonal $X \inj X \times X$ turns
$\Omega^*_G(X)$ into an $S(G)$-algebra. \\
$(vii) \ Projection \ formula :$ For a projective map $f : X' \to X$ in
$\sV^S_G$, one has for $x \in \Omega^G_*(X)$ and $x' \in \Omega^G_*(X')$,
the formula : $f_*\left(x' \cdot f^*(x)\right) = f_*(x') \cdot x$. \\
\end{thm}

\subsection{Formal group law}\label{subsection:FGL*}
Let $T$ be a split torus of rank $n$ acting on a smooth variety $X$.
Since $X$ is smooth, the commutative sub-$\bL$-algebra (under composition) of 
${\rm End}_{\bL}\left(\Omega^*_T(X)\right)$ generated
by the Chern classes of the vector bundles is canonically identified 
with a sub-$\bL$-algebra of the cobordism ring $\Omega^*_T(X)$ via the
identification $c^T_i(E) \mapsto \wt{c^T_i(E)} =
c^T_i(E)\left([X \xrightarrow{id} X]\right)$.
We shall denote this image also by $c^T_i(E)$ or, by $c^T_i$ if the underlying 
vector bundle is understood. In this paper, we shall view the (equivariant) 
Chern classes as elements of the (equivariant) cobordism ring of a smooth 
variety in this sense.

We recall from \cite{Krishna4} that the first Chern class of the tensor 
product of two equivariant line bundles on $X$ satisfies the formal group law 
of the ordinary cobordism. That is, for $L_1, L_2 \in {\rm Pic}^T(X)$, one has
\begin{equation}\label{eqn:FGL}
c^T_1(L_1 \otimes L_2) = F\left(c^T_1(L_1), c^T_1(L_2)\right),
\end{equation}
where $F(u, v) = u+v + uv{\underset{i,j \ge 1}\sum} \ a_{i,j}u^{i-1}v^{j-1},
\ a_{i,j} \in \bL_{1-i-j}$ is the graded power series in the graded power
series ring $\bL[[u,v]]$ which defines the universal (commutative) formal 
group law on $\bL$. It is known that even though $c^T_1(L)$ is not nilpotent 
in $\Omega^*_T(X)$ for $L \in {\rm Pic}^T(X)$ (unlike in the ordinary case), 
$F\left(c^T_1(L_1), c^T_1(L_2)\right)$ is a well defined element of 
$\Omega^1_T(X)$. In particular, there is a map of pointed sets
\begin{equation}\label{eqn:Chern-map}
\Pic^T(X) \to \Omega^1_T(X), \ \ L \mapsto c^T_1(L)
\end{equation}
such that $c^T_1(L_1 \otimes L_1) = F\left(c^T_1(L_1), c^T_1(L_2)\right)$.

One also knows that the formal group law $F(u,v)$ has
the following properties. \\
$(i)$ \ $F(u,0) = F(0,u) = u$ \\
$(ii)$ \ $F(u,v) = F(v,u)$ \\
$(iii)$ \ $F\left(u, F(v,w)\right) = F\left(F(u,v), w\right)$, and \\
$(iv)$ \ There exists (unique) $\rho(u) \in \bL[[u]]$ such that 
$F\left(u, \rho(u)\right) = 0$. \\
We write $F(u,v)$ as $u+_F v$. We shall denote $\rho(u)$ by $[-1]_Fu$.
We often write $F(u,v)$ as $u+_F v$. Inductively, we have 
$[n]_Fu = [n-1]_Fu +_F u$ if $n \ge 1$ and $[n]_Fu =
[-n]_F\rho(u)$ if $n \le 0$. The sum $\stackrel{m}{\underset{i =1}\sum}
[n_i]_Fu_i$ will mean $[n_1]_Fu_1 +_F \cdots +_F [n_m]_Fu_m$ for $n_i \in \Z$.

\subsection{Cobordism ring of classifying spaces}
Let $A = {\underset{j \in \Z}\oplus} A_j$ be a commutative graded
$R$-algebra with $R \subset A_0$ and $d \ge 0$. 
Let $S^{(n)} = {\underset{i \in \Z} \oplus} S_i$ be the graded ring such 
that $S_i$ is the set of formal power series in variables ${\bf t} =
\left(t_1, \cdots , t_n\right)$ of the form $f({\bf t}) = 
{\underset{m({\bf t}) \in \sC} \sum} a_{m({\bf t})}  m({\bf t})$,
where $a_{m({\bf t})}$ is a homogeneous element in $A$ of degree $|a_{m({\bf t})}|$
such that $|a_{m({\bf t})}| + |m({\bf t})| = i$. Here, $\sC$ is the set of all 
monomials in ${\bf t} = (t_1, \cdots , t_n)$ and 
$|m({\bf t})| = i_1 + \cdots + i_n$ if 
$m({\bf t}) = t^{i_1}_1 \cdots t^{i_n}_n$. 
One often writes this graded power series ring as ${A[[{\bf t}]]}_{\rm gr}$
to distinguish it from the usual formal power series ring $A[[{\bf t}]]$.

Notice that if $A$ is only non-negatively, then $S^{(n)}$ is nothing but the 
standard polynomial ring $A[t_1, \cdots , t_n]$ over $A$. It is also easy to 
see that $S^{(n)}$ is indeed a graded ring which is a subring of the formal 
power series ring $A[[t_1, \cdots , t_n]]$. The following result summarizes 
some basic properties of these rings. The proof is straightforward and is left 
as an exercise.

\begin{lem}\label{lem:GPSR}
$(i)$ There are inclusions of rings $A[t_1, \cdots , t_n] \subset S^{(n)} \subset
A[[t_1, \cdots , t_n]]$, where the first is an inclusion of graded rings. \\
$(ii)$ These inclusions are analytic isomorphisms with respect to the
${\bf t}$-adic topology. In particular, the induced maps of the associated
graded rings
\[
A[t_1, \cdots , t_n] \to {\rm Gr}_{({\bf t})} S^n \to 
{\rm Gr}_{({\bf t})} A[[t_1, \cdots , t_n]]
\]
are isomorphisms. \\
$(iii)$ ${S^{(n-1)}[[t_i]]}_{\rm gr} \xrightarrow{\cong} S^{(n)}$. \\
$(iv)$ $\frac{S^{(n)}}{(t_{i_1}, \cdots , t_{i_r})} \xrightarrow{\cong} S^{(n-r)}$ 
for any $n \ge r \ge 1$, where $S^{(0)} = A$. \\ 
$(v)$ The sequence $\{t_1, \cdots , t_n\}$ is a regular sequence in $S^{(n)}$.
\\
$(vi)$ If $A = R[x_1, x_2, \cdots ]$ is a polynomial ring with
$|x_i| < 0$ and ${\underset{i \to \infty}{\rm lim}} \ |x_i| = - \infty$, then
$S^{(n)} \xrightarrow{\cong}
{\underset{i}\varprojlim} \ {A[x_1, \cdots , x_i][[{\bf t}]]}_{\rm gr}$.
\end{lem}
 
Since we shall mostly be dealing with the graded power series rings in this 
text, we make the convention of writing ${A[[{\bf t}]]}_{\rm gr}$ as 
$A[[{\bf t}]]$, while the standard formal
power series ring will be written as $\widehat{A[[{\bf t}]]}$. We shall also 
write $S^{(n)}$ simply as $S$ when the number of ${\bf t}$-parameters is fixed. 

It is known \cite[Proposition~6.5]{Krishna4} that if $T$ is a split torus of 
rank $n$ and if $\{\chi_1, \cdots , \chi_n \}$ is a chosen basis 
of  the character group $\widehat{T}$, then there is a canonical isomorphism of
graded rings 
\begin{equation}\label{eqn:CBT*}
\bL[[t_1, \cdots , t_n]] \xrightarrow{\cong} \Omega^*(BT), \ \ 
t_i \mapsto c^T_1(L_{\chi_i}).
\end{equation}
Here, $L_{\chi}$ is the $T$-equivariant line bundle on ${\rm Spec}(k)$
corresponding to the character $\chi$ of $T$.
We shall write $S(T) = \Omega^*(BT)$ simply as $S$ in this text if the 
underlying torus is fixed.
One also has isomorphisms 
\begin{equation}\label{eqn:BT2}
\Omega^*(BGL_n) \xrightarrow{\cong} \bL[[\gamma_1, \cdots , \gamma_n]] \ \ 
{\rm and} \ \
\Omega^*(BSL_n) \xrightarrow{\cong} \bL[[\gamma_2, \cdots , \gamma_n]]
\end{equation}
of graded $\bL$-algebras, where $\gamma_i$'s  are the elementary symmetric
polynomials in $t_1, \cdots , t_n$ that occur in $\Omega^*(BT)$.

\subsection{Comparison with equivariant Chow groups}
In this paper, we fix the following notation for the tensor product
while dealing with inverse systems of modules over a commutative ring.
Let $A$ be a commutative ring with unit and let $\{L_n\}$ and $\{M_n\}$
be two inverse systems of $A$-modules with inverse limits $L$ and $M$
respectively. Following \cite{Totaro1}, one defines the 
{\sl topological tensor product} of $L$ and $M$ by
\begin{equation}\label{eqn:TTP}
L \widehat{\otimes}_A M : = 
{\underset{n}\varprojlim} \ (L_n {\otimes}_A M_n).
\end{equation} 
In particular, if $D$ is an integral domain with quotient field $F$ and if
$\{A_n\}$ is an inverse system of $D$-modules with inverse limit $A$, one 
has $A \widehat{\otimes}_{D} F = 
{\underset{n}\varprojlim} \ (A_n {\otimes}_{D} F)$.
The examples $\widehat{\Z_{(p)}} = {\underset{n}\varprojlim} \ {\Z}/{p^n}$ and 
$\Z[[x]] \otimes_{\Z} \Q \to {\underset{n}\varprojlim} \ 
\frac{\Z[x]}{(x^n)} {\otimes}_{\Z} \Q = \Q[[x]]$ show that the map  
$A \otimes_D F \to A \widehat{\otimes}_{D} F$ is in
general neither injective nor surjective. 

If $R$ is a $\Z$-graded ring and if $M$ and $N$ are two $R$-graded modules, 
then recall that $M\otimes_R N$ is also a graded $R$-module given by the 
quotient of $M \otimes_{R_0} N$ by the graded submodule generated by the 
homogeneous elements of the type $ax \otimes y - x \otimes ay$ where $a, x$ 
and $y$ are the homogeneous elements of $R$, $M$ and $N$ respectively.
If all the graded pieces $M_i$ and $N_i$ are the limits of inverse systems
$\{M^{\lambda}_i\}$ and $\{N^{\lambda}_i\}$ of $R_0$-modules, we define the
{\sl graded topological tensor product} as 
$M\widehat{\otimes}_R N = {\underset{i \in \Z}\bigoplus} 
\left(M\widehat{\otimes}_R N\right)_i$, where
\begin{equation}\label{eqn:TTP}
\left(M\widehat{\otimes}_R N\right)_i =
{\underset{\lambda}\varprojlim}
\left({\underset{j + j' = i}\bigoplus} \ 
\frac{M^{\lambda}_j \otimes_{R_0} N^{\lambda}_{j'}}
{\left(ax \otimes y - x \otimes ay\right)}\right).
\end{equation}
Notice that this reduces to the ordinary tensor product of graded $R$-modules
if the underlying inverse systems are trivial. 

It is known that there is a natural map $\Phi_X :\Omega^G_*(X) \to CH^G_*(X)$
of graded $\bL$-modules, where $CH^G_*(X)$ are the $G$-equivariant Chow
groups of $X$ defined by Totaro \cite{Totaro1} and Edidin-Graham \cite{EG}. 
The above map is in fact a morphism of $S(G)$-modules, which is a ring
homomorphism if $X$ is smooth. The following is the analogue of the 
corresponding result of Levine and Morel \cite{LM} for the equivariant
cobordism.
\begin{thm}\label{thm:CBCH}$(${\sl cf.} \cite[Proposition~7.1]{Krishna4}$)$
The map $\Phi_X$ induces an isomorphism of graded $\bL$-modules
\[
\Phi_X : \Omega^G_*(X) {\widehat{\otimes}}_{\bL} \Z \xrightarrow{\cong} CH^G_*(X).
\]
This is a ring isomorphism if $X$ is smooth.
\end{thm}

\section{Preliminary results}\label{section:PRELIMS}
In this section, we collect some preliminary results about the intersection
theory of the cobordism cycles and their equivariant analogues. We also recall
the notion of cohomological rigidity for a regular embedding of $G$-schemes
and its consequence for the structure of the equivariant cobordism ring
of smooth $G$-schemes.
 
\begin{lem}\label{lem:torind}
Let
\begin{equation}\label{eqn:Tor-ind1}
\xymatrix@C.9pc{
W \ar[r]^{g'} \ar[d]_{f'} & Z \ar[d]^{f} \\
Y \ar[r]_{g} & X}
\end{equation}
be a Cartesian square of closed immersions in $\sV^S_k$ such that $f$ and $g$
are transverse to each other. Then $g^* \circ f_* = f'_* \circ {g'}^*:
\Omega^*(Z) \to \Omega^*(Y)$.
\end{lem}
\begin{proof} 
Let $\alpha = [Z' \xrightarrow{p} Z]$ be a cobordism cycle on $Z$ such that
$Z'$ is smooth and irreducible and $p$ is projective. It suffices to show
that the result holds for $\alpha$ ({\sl cf.} ~\eqref{eqn:LP*}).
By \cite[Lemma~7.1]{LM1}, we can assume that $p$ is transverse to $g'$.
In particular, $f \circ p$ is transverse to $g$. Setting $W' = W \times_Z Z'$,
we see that $W'$ is smooth and there is a commutative diagram
\begin{equation}\label{eqn:Niv-PB1}
\xymatrix{
W' \ar[r]^{g''} \ar[d]_{p'} & Z' \ar[d]^{p} \\
W \ar[r]^{g'} \ar[d]_{f'} & Z \ar[d]^{f} \\
Y \ar[r]_{g} & X}
\end{equation}
where all the squares are Cartesian.
We now have 
\[
\begin{array}{lll}
g^* \circ f_*(\alpha) & = &  g^*\left([Z' \xrightarrow{f \circ p} X]\right) \\
& = & [W' \xrightarrow{f' \circ p'} Y] \\
& = & f'_*\left([W' \xrightarrow{p'} W]\right) \\
& = & f'_* \circ {g'}^*\left([Z' \xrightarrow{p} Z]\right) \\
& = & f'_* \circ {g'}^* (\alpha),
\end{array}
\]
where the second and the fourth equality hold by the transversality of
the vertical and the horizontal arrows in the bottom and top squares
({\sl cf.} \cite[5.1.3]{LM}). This completes the proof.
\end{proof}

\begin{cor}\label{cor:G-torind}
Let $G$ be a linear algebraic group over $k$ and let
\begin{equation}\label{eqn:Tor-ind1}
\xymatrix@C.9pc{
W \ar[r]^{g'} \ar[d]_{f'} & Z \ar[d]^{f} \\
Y \ar[r]_{g} & X}
\end{equation}
be a Cartesian square of closed immersions in $\sV^S_G$ such that $f$ and $g$
are transverse to each other. Then $g^* \circ f_* = f'_* \circ {g'}^*:
\Omega^*_G(Z) \to \Omega^*_G(Y)$.
\end{cor}
\begin{proof} We choose a sequence $\{(V_j, U_j)\}_{j}$ of $l_j$-dimensional
good pairs as in Theorem~\ref{thm:NO-Niveu} and for any $F \in \sV^S_G$,
let $F_j$ denote the mixed quotient $F \stackrel{G}{\times} U_j$.
We then see that 
\[
\xymatrix@C.9pc{
W_j \ar[r]^{g'_j} \ar[d]_{f'_j} & Z_j \ar[d]^{f_j} \\
Y_j \ar[r]_{g_j} & X_j}
\]
is a Cartesian square of closed immersions in $\sV^S_k$ such that $f_j$ and $g_j$
are transverse to each other. Applying Lemma~\ref{lem:torind}, we get
\begin{equation}\label{eqn:G-TI}
g_j^* \circ {f_j}_* = {f'_j}_* \circ {g'_j}^*:
\Omega^*_G(Z_j) \to \Omega^*_G(Y_j) \  {\rm for \ each} \ j \ge 0.
\end{equation}
Since the push-forward and pull-back maps on equivariant cobordism groups are
defined by taking the limit over these maps on the sequence of
mixed quotients, taking the inverse 
limit over $j \ge 0$ in ~\eqref{eqn:G-TI} and applying 
 Theorem~\ref{thm:NO-Niveu}, we get our desired assertion.
\end{proof} 

\begin{lem}\label{lem:RSIF}
Let $G$ be as above and let $Y \xrightarrow{f} X$ be a closed immersion
in $\sV^S_G$. Then $f_* \circ f^*(\alpha) = \alpha f_*(1)$ in $\Omega^*_G(X)$.
\end{lem}
\begin{proof}
Although this follows from the more general projection formula of 
Theorem~\ref{thm:Basic}, we give a simple proof of special case of closed
immersion. We first prove this in the non-equivariant case. As in the proof of
Lemma~\ref{lem:torind}, we can assume that $\alpha = [Z \xrightarrow{g} X]$,
where $g$ is transverse to $f$. In that case, we have
\[
f_* \circ f^*(\alpha) = [Z \times_X Y \to X] = 
\left([Z \xrightarrow{f} X]\right) \cdot
[Z \xrightarrow{g} X],
\]
where the second equality holds by the transversality condition and the
definition of intersection of cobordism cycles ({\sl cf.}
\cite[Sections~5,6]{LM}). The equivariant case is reduced to the 
non-equivariant one exactly as in the proof of Corollary~\ref{cor:G-torind}.
\end{proof}

The following is the equivariant version of the self-intersection
formula for cobordism. We refer the reader to \cite[Proposition~3.1]{Krishna3}
for a proof.

\begin{prop}[Self-intersection formula]\label{prop:SIF}
Let $G$ be as above and let $Y \xrightarrow{f} X$ be a regular 
$G$-equivariant embedding in $\sV_G$ of 
pure codimension $d$ and let $N_{Y/X}$ denote the equivariant normal bundle of 
$Y$ inside $X$. Then one has for every $y \in \Omega^G_*(Y)$,
$f^* \circ f_*(y) = c^G_d(N_{Y/X}) \cdot y$.
\end{prop}

\begin{lem}\label{lem:PROJ}
Let $G$ be as above and let $Y \in \sV^S_G$. Let $G$ act diagonally on $X =
Y \times \P^n_k$ by acting trivially on $\P^n_k$. There is a canonical
isomorphism 
\[
\stackrel{n}{\underset{p =0}\bigoplus} \ \Omega_G^{i-p}(Y) 
\to \Omega_G^i(X)
\]
\[
(a_0, \cdots , a_n) \mapsto \stackrel{n}{\underset{j =0}\sum}
{\pi}^*(a_p) \zeta^p
\]
where $\pi : X \to Y$ is the projection and $\zeta$ is the class of the 
tautological line bundle on $\P^n_k$. In particular, there is an
$S(G)$-algebra isomorphism
\[
\Omega^*_G(X) \xrightarrow{\cong} \Omega^*_G(Y) {\widehat{\otimes}}_{\bL} 
\Omega^*(\P^n_k).
\]
\end{lem}
\begin{proof}
We choose a sequence ${\{(V_j, U_j)\}}_j$ of good pairs for the $G$-action
as in Theorem~\ref{thm:NO-Niveu}. We then have
$X_j = X \stackrel{G}{\times} U_j = Y_j \times \P^n$ and hence the projective
bundle formula for the non-equivariant cobordism gives an isomorphism
\begin{equation}\label{eqn:PROJ1}
\stackrel{n}{\underset{p =0}\bigoplus} \ \Omega^{i-p}(Y_j) 
\to \Omega^i(X_j).
\end{equation}
Taking the inverse limit over $j \ge 0$ and using \cite[Theorem~6.1]{Krishna4},
we conclude the proof.
\end{proof} 

\subsection{Cohomological Rigidity in cobordism}\label{subsection:CRD}
Recall that for homomorphisms $A_i \xrightarrow{f_i} B, \ i= 1, 2$ of abelian 
groups, $A_1 {\underset{B}\times} A_2$ denotes 
the fiber product $\{(a_1, a_2)| f_1(a_1)$ \\
$ = f_2(a_2)\}$. 
Let $Y \inj X$ be a closed embedding in $\sV^S_G$ of codimension $d \ge 0$
and let $N_{Y/X}$ denote the normal bundle of $Y$ in $X$.
Following \cite{VV} and \cite{Krishna2}, we shall say that $Y$ is 
{\sl cohomologically rigid} inside $X$ if $c^G_d(N_{Y/X})$ is a non-zero 
divisor in the cobordism ring $\Omega^*_G(Y)$. The cohomological rigidity has 
the following important consequence. 

\begin{prop}\label{prop:split}
Let $Y \inj X$ be a closed embedding in $\sV^S_G$ such that $Y$ is 
cohomologically rigid inside $X$. Let $Y \stackrel{i}{\inj} X$ and
$U \stackrel{j}{\inj} X$ be the inclusion maps, where $U$ is the complement
of $Y$ in $X$. Then: \\
$(i)$ The localization sequence 
\[
0 \to \Omega^*_G(Y) \xrightarrow{i_*} \Omega^*_G(X) \xrightarrow{j^*}
\Omega^*_G(U) \to 0
\]
is exact. \\
$(ii)$ The restriction ring homomorphisms
\[
\Omega^*_G(X) \stackrel{(i^*, j^*)}{\longrightarrow} 
\Omega^*_G(Y) \times \Omega^*_T(U)
\]
give an isomorphism of rings
\[
\Omega^*_G(X) \xrightarrow{\cong} \Omega^*_G(Y) {\underset {\wt{\Omega^*_G(Y)}}
{\times}} \Omega^*_G(U),
\]
where $\wt{\Omega^*_G(Y)} = 
{\Omega^*_G(Y)}/{\left(c^G_d\left(N_{Y/X}\right)\right)}$,
and the maps 
\[
\Omega^*_G(Y) \to \wt{\Omega^*_G(Y)}, \ \Omega^*_G(U) \to \wt{\Omega^*_G(Y)}
\]
are respectively, the natural surjection and the map 
\[
\Omega^*_G(U) = {\frac{\Omega^*_G(X)}{i_*\left(\Omega^*_G(Y)\right)}}
\xrightarrow{i^*} {\frac{\Omega^*_G(Y)}{c^G_d\left(N_{Y/X}\right)}}  
= \wt{\Omega^*_G(Y)},
\]
which is well-defined by Proposition~\ref{prop:SIF}.
\end{prop}
\begin{proof}
{\sl Cf.} \cite[Proposition~4.1]{Krishna3}.
\end{proof}

\section{Stratification of toric varieties}\label{section:STRATA}
Recall that a split diagonalizable
group $G$ over $k$ is a commutative linear algebraic group whose identity
component is a split torus. 
Let $G$ be a split diagonalizable group acting on a smooth variety $X$.
Following the notations of \cite{VV}, for any $s \ge 0$, we let $X_{\le s}
\subset X$ be the open subset of points whose stabilizers have dimension
at most $s$. We shall often write $X_{\le s-1}$ also as $X_{< s}$.
Let $X_s = (X_{\le s} \setminus X_{< s})$ denote the locally closed subset of $X$,
where the stabilizers have dimension exactly $s$. We think of $X_s$
as a subspace of $X$ with the reduced induced structure. It is clear 
that $X_{\le s}$ and $X_s$ are $G$-invariant subspaces of $X$. Let $N_s$ 
denote the normal bundle of $X_s$ in $X_{\le s}$, and let $N_s^0$ denote the 
complement of the 0-section in $N_s$. Then $G$ clearly acts on $N_s$.
The following result of {\sl loc. cit.} describes some
properties of these subspaces.
\begin{prop}\label{prop:strata}
Let $s$ be non-zero integer. \\
$(i)$ There exists a finite number of $s$-dimensional subtori 
$T_1, \cdots , T_r$ in $G$ such that $X_s$ is the disjoint union of
the fixed point spaces $X^{T_j}_{\le s}$. \\
$(ii)$ $X_s$ is smooth locally closed subvariety of $X$. \\
$(iii) \ N^0_s = {\left(N_s\right)}_{< s}$. 
\end{prop}
\begin{proof} This is a special case of the more general
result \cite[Proposition~2.2]{VV} which holds for regular $G$-schemes over
any connected and separated Noetherian base scheme.
\end{proof}
Since $X$ is smooth, it follows from Proposition~\ref{prop:strata} that
$X_s$ is a smooth and closed subvariety of $X_{\le s}$, which is itself smooth.

In this paper, we are interested in applying Proposition~\ref{prop:strata}
to smooth toric varieties. For such varieties, the various strata $X_s$ have
simple description in terms of the torus orbits. Since this description
will be crucial for our results, we sketch it in some detail here and
refer the reader to \cite[Section~3]{Fulton1} for proofs.

So let $T$ be a split torus of rank $n$ over $k$. Let $M =
\Hom(\G_m, T)$ be its lattice of one-parameter subgroups and let 
$M^{\vee}$ be the character lattice of $T$. Let $\Delta$ be a 
fan in $M_{\R}$ and let $\Delta_1$ and $\Delta_{\rm max}$ denote the subsets of
the one-dimensional cones and the maximal cones in $\Delta$ respectively.

Let $X = X(\Delta)$ be the unique toric variety  
associated to the fan $\Delta$. We assume that every positive dimensional 
cone of $\Delta$ is generated by it edges such that the primitive vectors 
along these edges form a subset of a basis of $M$. This is 
equivalent to saying that $X$ is smooth. In this case, there is an one-to-one 
correspondence between the $T$-orbits in $X$ and the cones in $\Delta$. For 
every cone $\sigma \in \Delta$, the corresponding orbit $O_{\sigma}$ is 
isomorphic to the torus ${T}/{T_{\sigma}}$, where $T_{\sigma}$ is associated to the 
sublattice $M_{\sigma}$ of $M$ generated by $\sigma \cap M$.
Under this isomorphism, the origin (identity point) of ${T}/{T_{\sigma}}$
corresponds to the distinguished $k$-rational point $x_{\sigma}$ of $O_{\sigma}$.
In particular, for every $0 \le s \le n$, $X_s$ is of the form 
\begin{equation}\label{eqn:COMP}
X_s = {\underset{{\rm dim}(\sigma) = s}\coprod} \
O_{\sigma} \ \cong \  {\underset{{\rm dim}(\sigma) = s}\coprod} \ {T}/{T_{\sigma}}.
\end{equation}

We shall write $\tau \le \sigma$ if $\tau$ is a face of $\sigma$ as cones in
$\Delta$. The orbit closure $V_{\sigma}$ of $O_{\sigma}$ is the toric
variety associated to the fan 
$*({\sigma}) = \{\tau \in \Delta | \sigma \le \tau\}$, called the star of 
$\sigma$. Moreover, it is clear from the characterization of the 
smoothness of toric varieties that $V_{\sigma}$ is also smooth and is the union of
all orbits $O_{\tau}$ such that $\sigma$ is a face of $\tau$.
In particular, $O_{\sigma}$ is closed in $X$ if and only if $\sigma 
\in \Delta_{\rm max}$. The following general result will play a crucial role
in our study of the equivariant cobordism ring of toric varieties.

\begin{lem}\label{lem:rigiditysuff}
Let $G$ be a split diagonalizable group and let $T \subset G$ be a subtorus 
of rank $r \ge 1$. Let $G$ act on $X = (G/T) \times \P^n_k$
diagonally by acting trivially on $\P^n_k$ and via its natural action on
the quotient $G/T$. Let $E$ be a $G$-equivariant vector bundle of rank $d$ on 
$X$ such that in the eigenspace decomposition of $E$ with 
respect to $T$, the submodule corresponding to the trivial character is zero. 
Then $c^G_d(E)$ is a non-zero divisor in $\Omega^*_G(X)$.
\end{lem}
\begin{proof} 
Set $D = G/T$ and let $q : X = D \times \P^n   \to {\rm Spec}(k)$ be the 
structure map. We choose a splitting (not necessarily canonical) 
$G = T \times D$ and apply
\cite[Lemma~5.6]{Thomason3} to see that $E$ has a unique direct
sum decomposition 
\[
E = \ \stackrel{m}{\underset {i=1}{\oplus}} E_{\chi_i} \otimes q^*(L_{{\chi}_i}).
\]
Here, each $E_{\chi_i}$ is a $D$-bundle on $X$ and $L_{{\chi}_i}$ is the 
line bundle in ${\rm Pic}_T(k)$ corresponding to an 
$1$-dimensional representation ${\chi}_i$ of $T$. 
Since ${\rm rank}(E) = d$, we have by the Whitney sum formula,    
$c^G_d(E) = \stackrel{m}{\underset {i=1}{\prod}}c^G_{d_i}
\left(E_{{\chi}_i} {\otimes}L_{\chi_i}\right)$, where $d_i =
{\rm rank}(E_{\chi_i})$. We can thus assume that $E = 
E_{\chi} {\otimes} L_{\chi}$, where $\chi$ is not a trivial 
character of $T$ by our assumption. 

Since the $D$-equivariant vector bundles on $X$ are equivalent to ordinary
vector bundles on $\P^n$, we can identify $E_{\chi}$ with 
an ordinary vector bundle of rank $d$ on $\P^n$. Moreover,
we can apply Proposition~\ref{prop:Morita} and Lemma~\ref{lem:PROJ}
({\sl cf.} \cite[Lemma~6.1]{Krishna3}) to 
see that there is a ring isomorphism 
\[
\Omega^*_G(X) \cong \Omega^*_T(\P^n) \cong 
\Omega^*(\P^n)[[t_1, \cdots ,t_r]].
\]
Thus, we are reduced to showing that 
$c^T_d\left(E_{\chi} {\otimes} L_{\chi}\right)$ is a not a zero divisor 
in $\Omega^*(\P^n)[[t_1, \cdots ,t_r]]$ if $E_{\chi}$ is a vector bundle of rank
$d$ on $\P^n$ and $\chi$ is non-trivial character of $T$.

By \cite[Lemma~6.2]{Krishna3}, there is a morphism $Y \xrightarrow{p} \P^n$ 
which is a composition of the affine and projective bundles such that 
$p^*(E_{\chi})$ is a direct sum of line bundles and the $S$-algebra map 
$\Omega^*_T(\P^n) \xrightarrow{p^*} \Omega^*_T(Y)$ is injective.
We can thus replace $\P^n$ by $Y$ and assume that $E_{\chi}$ is a direct sum of
line bundles. Set $E_{\chi} = \stackrel{s}{\underset{i =1}\oplus} L_i$ and $v_i
= c_1(L_i)$. Since $\Omega^{> {\rm dim}(Y)}(Y) = 0$, we see that each $v_i$ is
nilpotent in $\Omega^*(Y)$.

We can write $f = c^T_1(L_{\chi}) = \stackrel{r}{\underset{j=1}\sum}
[m_j]_F t_j$ with $m_j \neq 0$ for some $j$. The Whitney sum formula then
yields $c^T_d(E_{\chi}) = \stackrel{s}{\underset{i =1}\prod}
F(f, v_i)$. Moreover, as $Y$ is obtained from $\P^n$ by a composition of
the affine and projective bundles, it follows from the projective bundle
formula and the homotopy invariance of the ordinary cobordism that
$\Omega^*(Y)$ is torsion-free. We now apply \cite[Lemma~5.3]{Krishna3}
with $R = \Z$ and $A = \Omega^*(Y)$ to conclude that $c^T_d(E_{\chi})$ is a
non-zero divisor in $\Omega^*_T(Y)$. 
\end{proof}

\section{Specialization maps in cobordism}\label{section:SPEC}
Our aim in the next two sections is to prove a decomposition theorem
for the equivariant cobordism rings of smooth toric varieties. 
This decomposition theorem
is motivated by the similar results for the equivariant $K$-theory in
\cite{VV} and for the equivariant Chow groups in \cite{Krishna2}.
As for the Chow groups and $K$-theory, we prove our decomposition theorem using
some specialization maps between certain equivariant cobordism rings.
We construct these maps in this section.
The main technical tool to do this is the {\sl deformation to 
the normal cone} method. Since this technique will be used at many steps
in the proofs, we briefly recall the construction from 
\cite[Chapter~5]{Fulton} for reader's convenience.

Let $X$ be a smooth $k$-scheme and let $Y \stackrel{f}{\inj} X$ be a
smooth closed subscheme of codimension $d \ge 1$. 
Let $\wt{M}$ be the blow-up of $X \times {\P}^1$ along $Y \times {\infty}$. 
Then $Bl_Y(X)$ is a closed subscheme of $\wt{M}$ and one denotes its
complement by $M$. There is a natural map ${\pi} : M \to {\P}^1$ such that
${\pi}^{-1}({\A}^1) \cong X \times {\A}^1$ with $\pi$ the projection map
and ${\pi}^{-1}(\infty) \cong X'$, where $X'$ is the total space of the
normal bundle $N_{Y/X}$ of $Y$ in $X$. One also gets the following 
diagram, where all the squares and the triangles commute.
\begin{equation}\label{eqn:DNC}
\xymatrix{
Y \ar[rd]_{u'} \ar[dd]_{f} \ar@/^{.5cm}/[rr]^{\ \ \ i_0} & & {Y \times {\P}^1} 
\ar[ll]^{p_Y}
\ar[dd]_{F} & Y \ar[l]_{\ \ \ i_{\infty}} \ar[dd]^{f'} \\
& {Y \times {\A}^1} \ar[dd]_{F'} \ar[ru]_{j'} & & \\
X \ar@/^{.4cm}/[rr]^{\ \ \ \ h} \ar[dr]_{u} & & M & X' \ar[l]^{i} \\
& {X \times {\A}^1}. \ar[ru]_{j} & & }
\end{equation}   
In this diagram, all the vertical arrows are the closed embeddings, $i_0$
and $i_{\infty}$ are the obvious inclusions of $Y$ in $Y \times {\P}^1$ along
the specified points, $i$ and $j$ are inclusions of the inverse
images of $\infty$ and ${\A}^1$ respectively under the map $\pi$, 
$u$ and $f'$ are are zero section embeddings and $p_Y$ is the projection
map. In particular, one has ${p_Y} \circ {i_0} = {p_Y} \circ {i_{\infty}}
= {id}_Y$.

We also make the observation here that in case $X$ is a $G$-variety
and $Y$ is $G$-invariant, then by letting $G$ act trivially on ${\P}^1$
and diagonally on $X \times {\P}^1$, one gets a natural action of $G$ on
$M$, and all the spaces in the above diagram become $G$-spaces and all the
morphisms become $G$-equivariant. This observation will be used in what 
follows. 
 
\subsection{Specialization maps}
Let $T$ be a split torus of rank $n$ and let $X = X(\Delta)$ be a smooth toric 
variety corresponding to a fan $\Delta$ in $M = \Hom(\G_m, T)$. 
For the rest of this paper, we fix a split torus $T$ and all toric varieties
will be associated to fans in the lattice $M$ of one-parameter subgroups in
$T$. There is a filtration of $X$ by $T$-invariant open
subsets 
\[
\emptyset = X_{\le -1} \subset X_{\le 0} \subset \cdots 
\subset X_{\le n} = X,
\]
where $X_{\le s}$ and $X_s$ are defined before.
In particular, $T$ acts freely on $X_{\le 0}$ and trivially
on $X_n$. In fact $X_{\le 0}$ is the dense orbit of $X$ isomorphic to $T$.
We fix $1 \le s \le n$ and let $X_s \stackrel{f_s}{\inj}
X_{\le s}$ and $X_{< s} \stackrel{g_s}{\inj} X_{\le s}$ denote the closed 
and the open embeddings respectively. Let ${\pi}: M_s \to {\P}^1$
be the deformation to the normal cone for the embedding $f_s$ as 
above. We have already observed that for the trivial action of 
$T$ on ${\P}^1$, $M_s$ has a natural $T$-action. Moreover,
the deformation diagram ~\eqref{eqn:DNC} is a diagram of smooth $T$-spaces.
For $0 \le t \le s$, we shall often denote the open subspace
${\left(M_s\right)}_{\le t}$ of $M_s$ by $M_{s, \le t}$. The terms like
$M_{s,t}$ and $M_{s, < t}$ (and also for $N_s$) will have similar
meaning in what follows. Since $T$ acts trivially on ${\P}^1$, it acts
on $M_s$ fiberwise, and one has $N_s = {\pi}^{-1}({\infty})$ and
\begin{eqnarray}\label{eqn:nolabel}
M_{s, \le t} \cap {\pi}^{-1}({\A}^1) = X_{\le t} \times {\A}^1 ; \
M_{s,t} \cap {\pi}^{-1}({\A}^1) = X_{t} \times {\A}^1.
\end{eqnarray}
Let $N_{s, \le t} \stackrel{i_{s, \le t}}{\inj} M_{s, \le t}$ and 
$X_{\le t} \times {\A}^1 \stackrel{j_{s,\le t}}{\inj} M_{s, \le t}$
denote the obvious closed and open embeddings. We define $i_{s,t}$ and 
$j_{s,t}$ similarly. Let $N_{s,t} \stackrel{{\eta}_{s,t}}{\inj}
N_{s, \le t}$ and $M_{s,t} \stackrel{{\delta}_{s,t}}{\inj} M_{s,\le t}$
denote the other closed embeddings. 
One has a commutative diagram
\begin{equation}\label{eqn:DNC*}
\xymatrix{
X_{\le t} \ar[r]^{g_{\le t}} \ar[d]_{f_{s, \le t}} & 
{X_{\le t} \times {\A}^1}
\ar[r]^{\ j_{s, \le t}} \ar[d] & M_{s, \le t} \ar[d] & & \\
X_{\le s} \ar[r]^{g_{0, \le s}} & X_{\le s} \times {\A}^1 \ar[r]^{\ \ j_{\le s}}
& M_s \ar[r] & X_{\le s} \times {\P}^1 \ar[r] &  X_{\le s},}
\end{equation}  
where $g_{\le t}$ is the 0-section embedding, and the composite of
all the maps in the bottom row is identity.
This gives us the diagram of equivariant cobordism groups
\begin{equation}\label{eqn:DNC*1}
\xymatrix{
\Omega^*_T\left(N_{s,t}\right) \ar[r]^{{{i_{s,t}}_*}} 
\ar[d]^{{{\eta}_{s,t}}_*} & 
\Omega^*_T\left(M_{s,t}\right) \ar[r]^{j_{s,t}^*} \ar[d]^{{{\delta}_{s,t}}_*}  & 
\Omega^*_T\left(X_{t} \times {\A}^1 \right) \ar[r]^{{g_{t}^*}} 
\ar[d]^{{f_t}_*} & \Omega^*_T\left(X_{t}\right) \ar[d]^{{f_t}_*} \\
\Omega^*_T\left(N_{s, \le t}\right) \ar[r]_{{i_{s, \le t}}_*} &
\Omega^*_T\left(M_{s, \le t}\right) \ar[r]_{j_{s, \le t}^*} &
\Omega^*_T\left(X_{s, \le t} \times {\A}^1 \right) 
\ar[r]_{\hspace*{.5cm} {g_{\le t}^*}} & 
\Omega^*_T\left(X_{\le t}\right),}
\end{equation}
where the left and the middle squares commute by Theorem~\ref{thm:Basic} $(i)$ 
and the right square commutes by Corollary~\ref{cor:G-torind}.
Since the last horizontal maps in both rows are natural isomorphisms by the 
homotopy invariance, we shall often identify the last two terms in both
rows and use ${j_{s, \le t}^*}$ and 
${\left(j_{s, \le t} \circ g_{\le t}\right)}^*$ interchangeably.

\begin{lem}\label{lem:STR}
For every $1 \le s \le n$ and $0 \le t \le s$, one has
\[
N_{s, t} \cong {\underset{{\rm dim}(\sigma) = s}{\underset{\sigma \in \Delta}
\coprod}} \ \
{\underset{{\rm dim}(\tau) = t}{\underset{\tau \le \sigma}\coprod} O_{\tau}} \ \
{\rm and} \ \
M_{s, t} \cong {\underset{{\rm dim}(\sigma) = s}{\underset{\sigma \in \Delta}
\coprod}} \ \
{\underset{{\rm dim}(\tau) = t}{\underset{\tau \le \sigma}\coprod} \left(O_{\tau}
\times \P^1\right)}.
\]
\end{lem}
\begin{proof}
This follows from \cite[Lemma~6.1]{VV} by very little 
modifications. We only give the sketch.
We first assume that the fan $\Delta$ consists of all the faces of a cone
$\sigma$ of dimension $s \ge 1$ such that $X_{\le s} = U_{\sigma}$. In this case, 
$X_{\le s}$ is a $T$-equivariant vector bundle over $O_{\sigma}$ in such a way 
that the zero-section is precisely the inclusion 
$O_{\sigma} \subset X_{\le s}$. In 
particular, one has $X_{\le s} \cong N_s$ and $M_{\le s} \cong X_{\le s} \times
\P^1$. We now see from ~\eqref{eqn:COMP} that 
\[
N_{s,t} \cong {\underset{{\rm dim}(\tau) = t}{\underset{\tau \le \sigma}\coprod} 
O_{\tau}} \ \ {\rm and} \ \
M_{s, t} \cong
{\underset{{\rm dim}(\tau) = t}{\underset{\tau \le \sigma}\coprod} \left(O_{\tau}
\times \P^1\right)}.
\]   
For the general case, we notice that $X_{s} = 
{\underset{{\rm dim}(\sigma) = s}\coprod} O_{\sigma}$ by ~\eqref{eqn:COMP} and
for an $s$-dimensional cone $\sigma \in \Delta$, the intersection of
$X_s$ with $U_{\sigma}$ is precisely $O_{\sigma}$. Hence, we reduce the special case
above.
\end{proof}

\begin{cor}\label{cor:rigidstrata}
Let $X = X(\Delta)$ be a smooth toric variety. Then for $s \ge 1$, $X_s$ is 
cohomologically rigid inside $X_{\le s}$. The same holds for the inclusions
$N_{s,t} \subset N_{s, \le t}$ and $M_{s,t} \subset M_{s, \le t}$.
\end{cor}
\begin{proof}
Let $d_s$ be the codimension of $X_s$ inside $X_{\le s}$. By ~\eqref{eqn:COMP},
it suffices to show that $O_{\sigma} = {T}/{T_{\sigma}}$ is cohomologically
rigid inside $X_{\le s}$, where $\sigma$ is a cone of $\Delta$ of dimension $s$.
Since $s \ge 1$, it follows from the parts $(i)$ and $(iii)$ of 
Proposition~\ref{prop:strata} that the conditions of 
Lemma~\ref{lem:rigiditysuff} are satisfied for $D = O_{\sigma}$ and $n =0$. 
We conclude that $O_{\sigma}$ is cohomologically rigid inside $X_{\le s}$. 
The conclusions for $N_{s,t}$ and $M_{s,t}$ follows in the same way using
Lemma~\ref{lem:STR} and Lemma~\ref{lem:rigiditysuff}.
\end{proof}

\begin{thm}\label{thm:specialization}
There are ring homomorphisms 
\[
{{Sp}}_{X, s}^{\le t} : \Omega^*_T\left(X_{\le t}\right)
\to \Omega^*_T\left(N_{s,\le t}\right);
\]
\[ 
{{Sp}}_{X, s}^{t} : \Omega^*_T\left(X_{t}\right)
\to \Omega^*_T\left(N_{s, t}\right)
\]
such that $i_{s,\le t}^* = {{Sp}}_{X, s}^{\le t} \circ
{j_{s, \le t}^*}$ and $i_{s,t}^* = {Sp}_{X, s}^{t} \circ
{j_{s,t}^*}$. Moreover, both the squares in the diagram
\begin{equation}\label{eqn:main}
\xymatrix{
\Omega^*_T\left(X_{\le t}\right) \ar[r]^{f_t^*}   
\ar[d]_{{Sp}_{X, s}^{\le t}} & \Omega^*_T\left(X_{t}\right)
\ar[d]^{{Sp}_{X, s}^{t}} \ar[r]^{{f_t}_*} &
\Omega^*_T\left(X_{\le t}\right) \ar[d]^{{Sp}_{X, s}^{\le t}} \\
\Omega^*_T\left(N_{s,\le t}\right) \ar[r]_{{\eta}_{s,t}^*} &
\Omega^*_T\left(N_{s,t}\right) \ar[r]_{{{\eta}_{s,t}}_*} &
\Omega^*_T\left(N_{s,\le t}\right)}
\end{equation} 
commute.
\end{thm}   
\begin{proof}
We consider the commutative diagram
\begin{equation}\label{eqn:special}
\xymatrix{
0 \ar[d] & 0 \ar[d] & 0 \ar[d] & \\
\Omega^*_T\left(N_{s,t}\right) 
\ar[r]^{{{i_{s,t}}_*}} \ar[d]^{{{\eta}_{s,t}}_*} &
\Omega^*_T\left(M_{s,t}\right) 
\ar[r]^{j_{s,t}^*} \ar[d]^{{{\delta}_{s,t}}_*} & 
\Omega^*_T\left(X_{t}\right) \ar[r] \ar[d]^{{f_t}_*} & 0 \\
\Omega^*_T\left(N_{s,\le t}\right) 
\ar[r]^{{{i_{s, \le t}}_*}} \ar[d] &
\Omega^*_T\left(M_{s,\le t}\right) 
\ar[r]^{j_{s, \le t}^*} \ar[d] & 
\Omega^*_T\left(X_{\le t}\right) \ar[r] \ar[d] & 0 \\
\Omega^*_T\left(N_{s,\le t-1}\right) 
\ar[r]^{{{i_{s, \le t-1}}_*}} \ar[d] &
\Omega^*_T\left(M_{s,\le t-1}\right) 
\ar[r]^{j_{s, \le t-1}^*} \ar[d] & 
\Omega^*_T\left(X_{\le t-1}\right) \ar[r] \ar[d] & 0. \\ 
0 & 0 & 0 & }
\end{equation}
All rows in this diagram are exact by Theorem~\ref{thm:Basic} $(ii)$, and
all columns are exact by Corollary~\ref{cor:rigidstrata} and 
Proposition~\ref{prop:split}. 

We apply Proposition~\ref{prop:SIF}
to the inclusions $i_{s,\le t}$ and $i_{s,t}$ to see that the composites
$i_{s,\le t}^* \circ {i_{s,\le t}}_*$ and $i_{s,t}^* \circ {i_{s,t}}_*$ are
multiplication by the first Chern class $c^T_1$ of the corresponding
normal bundles. But these normal bundles are the inverse images of a
line bundle on $\P^1$. It follows that these 
normal bundles are trivial, because the restriction of 
any line bundle on $\P^1$ to $\infty \in \P^1$ and hence on the
fiber over $\infty$ is clearly trivial.  
We conclude that the composites $i_{s,\le t}^* \circ {i_{s,\le t}}_*$ and 
$i_{s,t}^* \circ {i_{s,t}}_*$ are zero. 

The above diagram now
automatically defines the specializations ${Sp}_{X, s}^{\le t}$
and ${Sp}_{X, s}^{t}$ and gives the desired factorization
of $i_{s,\le t}^*$ and $i_{s,t}^*$. Since $i_{s,t}^*$ and
$j_{s,t}^*$ are ring homomorphisms, and since the latter is
surjective, we deduce that ${Sp}_{X, s}^{t}$ is also a ring homomorphism. 
The map ${Sp}_{X, s}^{\le t}$ is a ring homomorphism for the same reason.

We are now left with the proof of the commutativity of ~\eqref{eqn:main}.  
To prove that the right square commutes, we consider the following
diagram.
\begin{equation}\label{eqn:nospecial}
\xymatrix@C.7pc{
\Omega^*_T\left(M_{s,t}\right) \ar@{->>}[drr]^{j_{s,t}^*} 
\ar[rrr]^{{{\delta}_{s,t}}_*} 
\ar[dd]_{i_{s,t}^*} & & & 
\Omega^*_T\left(M_{s,\le t}\right) \ar@{->>}@/^/[drr]^{\ \ \ \ j_{s,\le t}^*} 
\ar[dd]_{i_{s,\le t}^*} & & \\
& & \Omega^*_T\left(X_{t}\right) \ar[lld]_{{Sp}_{X, s}^{t}} 
\ar@/_1pc/[rrr]_{\ \ \ \ \ \ \ {f_t}_*} & & & 
\Omega^*_T\left(X_{\le t}\right) \ar@/^1pc/[lld]^{{Sp}_{X, s}^{\le t}} \\
\Omega^*_T\left(N_{s, t}\right) \ar[rrr]^{{{\eta}_{s,t}}_*} & & &
\Omega^*_T\left(N_{s, \le t}\right) & & }
\end{equation}
It is easy to check that $N_{s, \le t}$ and $M_{s,t}$ are 
transverse over $M_{s,\le t}$ ({\sl cf.} \cite[Lemma~1]{VV1})
and hence the back face of the above
diagram commutes by Corollary~\ref{cor:G-torind}. The upper face commutes
by diagram ~\eqref{eqn:special}. Since $j_{s,t}^*$ is surjective, a
diagram chase shows that the lower face also commutes, which is what
we needed to prove.

Finally, since we have shown in diagram ~\eqref{eqn:special} that
${{{\eta}_{s,t}}_*}$ is injective and since  
the right square commutes, it now suffices to show that
the composite square in ~\eqref{eqn:main} commutes in order to show that
the left square there commutes. 

By Lemma~\ref{lem:RSIF}, the composite maps ${f_t}_* \circ f_t^*$
and ${{{\eta}_{s,t}}_*} \circ {\eta}_{s,t}^*$ are multiplication by
${f_t}_* (1)$ and ${{{\eta}_{s,t}}_*}(1)$ respectively. Since 
${Sp}_{X, s}^{\le t}$ and ${Sp}_{X, s}^{t}$ are ring
homomorphisms, it suffices to show that 
\[
{Sp}_{X, s}^{\le t}\left({f_t}_* \circ j_{s,t}^* (1)\right) =
{Sp}_{X, s}^{\le t}\left({f_t}_* (1)\right) = {{{\eta}_{s,t}}_*}(1).
\]
But this follows directly from the commutativity of the right square in
~\eqref{eqn:main}.
\end{proof}  

\section{Decomposition for equivariant cobordism ring}
\label{section:DECOMP}
We need the following intermediate step for the decomposition theorem
for the equivariant cobordism rings of smooth toric varieties. 
Let $X = X(\Delta)$ be a smooth toric variety as above.
\begin{prop}\label{prop:decomposition1}
The restriction maps 
\[
\Omega^*_T\left(X_{\le s}\right) \stackrel{(f_s^*, g_s^*)}{\longrightarrow}
\Omega^*_T\left(X_{s}\right) \times \Omega^*_T\left(X_{< s}\right)
\]
define an isomorphism of rings
\[
\Omega^*_T\left(X_{\le s}\right) \xrightarrow{\cong}
\Omega^*_T\left(X_{s}\right) {\underset {\Omega^*_T\left(N_s^0 \right)}
{\times}} \Omega^*_T\left(X_{< s}\right),
\]
where $\Omega^*_T\left(X_{s}\right) \xrightarrow{{\eta}_{s,\le s-1}^*}
\Omega^*_T\left(N_s^0 \right)$
is the pull-back
\[
\Omega^*_T\left(X_{s}\right) \xrightarrow{\cong} 
\Omega^*_T\left(N_{s}\right) \to \Omega^*_T\left(N_s^0 \right)
\]
and 
\[
\Omega^*_T\left(X_{< s}\right) \stackrel{{{Sp}_{X, s}^{\le s-1}}}
{\longrightarrow} \Omega^*_T\left(N_{s, \le s-1}\right) =
\Omega^*_T\left(N_s^0 \right)
\]
is the specialization map of Theorem~\ref{thm:specialization}.
\end{prop} 
\begin{proof} We only need to identify the pull-back and the specialization
maps with the appropriate maps of Proposition~\ref{prop:split}.
In the diagram 
\[
\xymatrix{
0 \ar[r] & \Omega^*_T\left(X_{s}\right) \ar[r]^{{f_{s, \infty}}_*} 
\ar[dr]_{c^G_{d_s}} & 
\Omega^*_T\left(N_{s}\right) \ar[r]^{{\eta}_{s, \le s-1}^*} 
\ar[d]^{f_{s, \infty}^*} &
\Omega^*_T\left(N_s^0 \right) \ar[r] & 0 \\
& & \Omega^*_T\left(X_{s}\right) & & }
\]
where $f_{s, \infty}: X_s \to N_s$ is the 0-section embedding, the top
sequence is exact by Proposition~\ref{prop:strata} and
Corollary~\ref{cor:rigidstrata}, and the lower
triangle commutes by Proposition~\ref{prop:SIF}. Since ${f_{s, \infty}^*}$
is an isomorphism, this immediately identifies the pull-back map
of the proposition with the quotient map
$\Omega^*_T\left(X_{s}\right) \to \frac{\Omega^*_T\left(X_{s}\right)}
{\left({c^T_{d_s}(N_s)}\right)}$.

Next, we consider the diagram
\[
\xymatrix{
\Omega^*_T\left(X_{\le s}\right) \ar[d]_{{{{Sp}_{X, s}^{\le s}}}}
\ar@{->>}[r]^{f_{s,\le s-1}^*} & \Omega^*_T\left(X_{< s}\right)
\ar[d]^{{{{Sp}_{X, s}^{\le s-1}}}} \\
\Omega^*_T\left(N_{s}\right) \ar[r] \ar[d]_{f_{s, \infty}^*} &
\Omega^*_T\left(N_s^0 \right) \\
\Omega^*_T\left(X_{s}\right) \ar[ur]_{{\eta}_{s, \le s-1}^*}. & }
\]
Since the top horizontal arrow in the above diagram is surjective,
we only need to show that
${{{{Sp}_{X, s}^{\le s-1}}}} \circ {f_{s,\le s-1}^*}
= {{\eta}_{s, \le s-1}^*} \circ f_s^*$ in order to identify 
${{Sp}_{X, s}^{\le s-1}}$ with the map $j^*$ of
Proposition~\ref{prop:split}.
It is clear from the diagram ~\eqref{eqn:special} and the definition of the 
specialization maps that the top square above commutes. We have just
shown above that the lower triangle also commutes. This reduces us to
showing that 
\begin{equation}\label{eqn:special2}
{f_{s, \infty}^*} \circ {{{{Sp}_{X, s}^{\le s}}}}
= f_s^*.
\end{equation}
If $X_s \times {\P}^1 \xrightarrow{F_s} M_s$ denotes the embedding
({\sl cf.} ~\eqref{eqn:DNC}), then for $x \in \Omega^*_T\left(X_{\le s}\right)$,
we can write $x = j_{\le s}^*(y)$ by the surjectivity of $j^*_{\le s}$
(see the proof of Theorem~\ref{thm:specialization}).
Then 
\[
\begin{array}{lllll}
{f_{s, \infty}^*} \circ {{{{Sp}_{X, s}^{\le s}}}} \circ
 j_{\le s}^*(y) & = & {f_{s, \infty}^*} \circ i_{s, \le s}^*(y) &
= & g_{\infty, \le s}^* \circ F_s^*(y) = g_{0, \le s}^* \circ F_s^*(y) \\   
& = & f_s^* \circ j_{\le s}^*(y) & = & f_s^*(x),
\end{array}
\]
where the second equality follows from Corollary~\ref{cor:G-torind}.
This proves ~\eqref{eqn:special2} and the proposition.
\end{proof}

\begin{thm}\label{thm:decomposition*}
For a smooth toric variety $X = X(\Delta)$, the ring homomorphism
\[
\Omega^*_T\left(X \right) \longrightarrow 
\stackrel{n}{\underset{s=0}{\prod}} \Omega^*_T\left(X_s \right)
\]
is injective. Moreover, its image consists of the $n$-tuples
$\left({\alpha}_s\right)$ in the product with the property that
for each $s = 1, \cdots , n$, the pull-back of ${\alpha}_s \in
\Omega^*_T\left(X_s \right)$ in $\Omega^*_T\left(N_{s, s-1} \right)$ is same
as ${{Sp}_{X, s}^{s-1}}\left({\alpha}_{s-1}\right) \in
\Omega^*_T\left(N_{s, s-1} \right)$.
In other words, there is a ring isomorphism
\[
\Omega^*_T\left(X \right) \stackrel{\cong}{\longrightarrow}
\Omega^*_T\left(X_n \right) {\underset{\Omega^*_T\left(N_{n,n-1} \right)}
{\times}} \Omega^*_T\left(X_{n-1} \right) 
{\underset{\Omega^*_T\left(N_{n-1,n-2} \right)}
{\times}} \cdots {\underset{\Omega^*_T\left(N_{1,0} \right)}
{\times}} \Omega^*_T\left(X_0 \right).
\]
\end{thm}
\begin{proof} We prove by the induction on the largest integer $s$ such that
$X_s \neq \emptyset$. 

If $s = 0$, there is nothing to prove. If $s > 0$, we have by induction
\begin{equation}\label{eqn:decomposition*1}
\Omega^*_T\left(X_{<s} \right) \stackrel{\cong}{\longrightarrow}
\Omega^*_T\left(X_{s-1} \right) {\underset{\Omega^*_T\left(N_{s-1,s-2} \right)}
{\times}} \cdots {\underset{\Omega^*_T\left(N_{1,0} \right)}
{\times}} \Omega^*_T\left(X_0 \right).
\end{equation}
Using this and Proposition~\ref{prop:decomposition1}, it suffices to
show that if ${\alpha}_s \in \Omega^*_T\left(X_s \right)$ and
if ${\alpha}_{<s} \in \Omega^*_T\left(X_{<s} \right)$ with the restriction
${\alpha}_{s-1} \in \Omega^*_T\left(X_{s-1} \right)$ are such that
${\alpha}_s \mapsto {\alpha}_s^0 \in \Omega^*_T\left(N_s^0 \right)$
and ${\alpha}_s \mapsto {\alpha}_{s, s-1} \in     
\Omega^*_T\left(N_{s,s-1} \right)$, then
\[
{{Sp}_{X, s}^{\le s-1}}\left({\alpha}_{<s}\right) = 
{\alpha}_s^0 \ \ {\rm iff} \ \
{{Sp}_{X, s}^{s-1}}\left({\alpha}_{s-1}\right) = 
{\alpha}_{s, s-1}.
\]
Using the commutativity of the left square in 
Theorem~\ref{thm:specialization},
this is reduced to showing that the restriction map
\begin{equation}\label{eqn:spec*1}
\Omega^*_T\left(N_s^0 \right) \to \Omega^*_T\left(N_{s,s-1} \right)
\end{equation}
is injective. 

It follows from ~\eqref{eqn:COMP} and Lemma~\ref{lem:STR} that  
\[
N_s = {\underset{{\rm dim}(\sigma) = s}{\underset{\sigma \in \Delta}
\coprod}} N_{{O_{\sigma}}/{X_{\le s}}} \ \ {\rm and} \ \
N_{s, s-1} = {\underset{{\rm dim} (\sigma) = s}{\underset{\sigma \in \Delta}
\coprod}} {\underset{\tau \in \partial \sigma}\coprod} O_{\tau}.
\]
Moreover, for every cone $\sigma \in \Delta$, we have 
$N_{{O_{\sigma}}/{X_{\le s}}} = N_{{O_{\sigma}}/{U_{\sigma}}}$, where $U_{\sigma}$ is the open
toric subvariety of $X$ defined by the fan consisting of all faces of
$\sigma$. Thus it suffices to prove ~\eqref{eqn:spec*1}
when $X$ is of the form $U_{\sigma}$, where $\sigma$ is an $s$-dimensional cone
in $\Delta$.

In this case, $U_{\sigma} \cong N_s$ is the $T$-equivariant vector bundle 
over $O_{\sigma} = T/{T_{\sigma}}$ of the form $(T \times V_{\sigma})/{T_{\sigma}} \to
T/{T_{\sigma}}$, where $V_{\sigma}$ is the $k$-vector space spanned by
the part $B$ of a basis of $M$ which generates $\sigma$. In particular,
$N_s = \stackrel{s}{\underset{j = 1}\oplus}  L_{\chi_j}$, where 
$\{\chi_1, \cdots , \chi_s\}$ is a part of a basis of $M$ and $L_{\chi}$ is the
$T$-equivariant line bundle on $O_{\sigma}$ associated to the character $\chi$
of $T$.

Next, it follows from Proposition~\ref{prop:decomposition1} that
\[
{\rm Ker}\left(\Omega^*_T\left(X_{s} \right) \to 
\Omega^*_T\left(N_{s,s-1} \right)\right) = \
\stackrel{s}{\underset{i=1}{\bigcap}} \left(c^T_{1}(L_{\chi_i})\right) 
\ \ {\rm and} 
\]   
\[
{\rm Ker}\left(\Omega^*_T\left(X_{s} \right) \to 
\Omega^*_T\left(N_{s}^0 \right)\right) = \left(c^T_{s}(N_s)\right).
\]
Setting ${\gamma}_i = c^T_{1}(L_{\chi_i})$ and ${\gamma} = 
c^T_{s}(N_s)$, we see from the surjectivity 
$\Omega^*_T\left(X_{s} \right) \surj \Omega^*_T \left(N_{s}^0 \right)$ that
showing the injectivity of the map in ~\eqref{eqn:spec*1} is equivalent to 
showing that  
\begin{equation}\label{eqn:elem2}
\left({\gamma}\right) = 
\left(\stackrel{s}{\underset{i=1}{\prod}} {\gamma}_i \right) 
= \ \stackrel{s}{\underset{i=1}{\bigcap}} \left({\gamma}_i\right)
\end{equation}
in $\Omega^*_T(X_s) \cong \Omega^*_{T_{\sigma}}(k) = \bL[[t_1, \cdots, t_s]]$.
Since $B = \{\chi_1, \cdots  , \chi_s\}$ can be identified with a basis of
$\widehat{T_{\sigma}}$, we have an isomorphism
$ \Omega^*_{T_{\sigma}}(k) \cong \bL[[t_1, \cdots, t_s]]$, where 
$t_i = c^T_1(L_{\chi_i})$ for $1 \le i \le s$ ({\sl cf.} ~\eqref{eqn:CBT*}). 
The equality of ~\eqref{eqn:elem2} is now easy to check and 
follows also from \cite[Lemma~5.4]{Krishna3}.
\end{proof}

\begin{lem}\label{lem:VV1}
For $1 \le s \le n$, there is a canonical isomorphism
\[
N_{s, s-1} = {\underset{{\rm dim}(\sigma) = s}{\underset{\sigma \in \Delta}
\coprod}} {\underset{\tau \in \partial \sigma}\coprod} O_{\tau}.
\]
Furthermore, for each $s$-dimensional cone $\sigma$ and $\tau \in \partial 
\sigma$, the composition of the map
\[
{{Sp}_{X, s}^{s-1}} : \Omega^*_T\left(X_{s-1}\right) =
{\underset{{\rm dim}(\tau) = s-1}{\underset{\tau \in \Delta}
\prod}}  \Omega^*_T\left(O_{\tau}\right) \ \to \ \Omega^*_T\left(N_{s,s-1}\right) =
{\underset{{\rm dim}(\sigma) = s}{\underset{\sigma \in \Delta}
\prod}} {\underset{\tau \in \partial \sigma}\prod}
\Omega^*_T\left(O_{\tau}\right)
\]
with the projection 
\[
{\rm Pr}_{\sigma, \tau} : 
{\underset{{\rm dim}(\sigma) = s}{\underset{\sigma \in \Delta}
\prod}} {\underset{\tau \in \partial \sigma}\prod}
\Omega^*_T\left(O_{\tau}\right) \ \to \ \Omega^*_T\left(O_{\tau}\right) 
\]
is the projection 
\[
{\underset{{\rm dim}(\tau) = s-1}{\underset{\tau \in \Delta}
\prod}}  \Omega^*_T\left(O_{\tau}\right) \ \to \ 
\Omega^*_T\left(O_{\tau}\right).
\]
\end{lem}
\begin{proof}
The first assertion is already shown in Lemma~\ref{lem:STR}.
The second assertion for the $K$-theory is shown in \cite[Lemma~6.1]{VV}.
The same proof goes through here as well in verbatim in view of
our description of the specialization maps in 
Theorem~\ref{thm:specialization}.
\end{proof}
The following is our first result in the description of the equivariant
cobordism ring of a smooth toric variety. This is analogous to the
similar description of the equivariant $K$-theory in \cite[Theorem~6.2]{VV} and
equivariant Chow groups in \cite[Theorem~5.4]{Brion}. Recall that
$S$ denotes the ring $S(T) = \Omega^*_T(k)$. 

\begin{thm}\label{thm:TORIC-I}
Let $X = X(\Delta)$ be a smooth toric variety associated to a fan $\Delta$
in $M_{\R}$. There is an injective homomorphism of $S$-algebras
\[
\Phi_X : \Omega^*_T(X) \to {\underset{\sigma \in \Delta_{\rm max}}\prod} 
S\left(T_{\sigma}\right),
\]
where $\Delta_{\rm max}$ is the set of maximal cones in $\Delta$.
An element $\left(a_{\sigma}\right) \in 
{\underset{\sigma \in \Delta_{\rm max}}\prod} 
S\left(T_{\sigma}\right)$ is in the image of this homomorphism if and only
if for any two maximal cones $\sigma_1$ and $\sigma_2$, the restrictions of
$a_{\sigma_1}$ and $a_{\sigma_2}$ to $S\left(T_{\sigma_1 \cap \sigma_2}\right)$ coincide.
\end{thm}
\begin{proof}
Since every proper face of a cone $\sigma$ is contained in $\partial \sigma$,
it follows immediately from \eqref{eqn:COMP}, 
Theorem~\ref{thm:decomposition*}, Lemma~\ref{lem:VV1} and the isomorphism
$\Omega^*_T\left(O_{\sigma}\right) \cong S\left(T_{\sigma}\right)$ that
$\Omega^*_T(X)$ is a subring of ${\underset{\sigma \in \Delta}\prod} \
S\left(T_{\sigma}\right)$ consisting of elements $\left(a_{\sigma}\right)$
with the property that the restriction of $a_{\sigma} \in S\left(T_{\sigma}\right)$ 
to $S\left(T_{\tau}\right)$ coincides with $a_{\tau}$ whenever $\tau \le \sigma$.
The theorem now follows from the fact that every cone in $\Delta$ is contained
in a maximal cone in $\Delta$.
\end{proof}

If $X = X(\Delta)$ is a smooth projective toric variety, then all the maximal
cones in $\Delta$ are $n$-dimensional and the closed orbits $O_{\sigma}$ are
all $k$-rational points. We conclude from Theorem~\ref{thm:TORIC-I} that

\begin{cor}\label{cor:Proj-T}
For a smooth projective toric variety $X = X(\Delta)$, there is an injective
homomorphism of $S$-algebras
\[
\Omega^*_T(X) \inj  {\underset{\sigma \in \Delta_{\rm max}}\prod} 
S \cong S^{|\Delta_{\rm max}|}
\]
whose image is the set of elements $(a_{\sigma})$ such that for any two adjacent
maximal cones $\sigma_1$ and $\sigma_2$, the restrictions of
$a_{\sigma_1}$ and $a_{\sigma_2}$ to $S\left(T_{\sigma_1 \cap \sigma_2}\right)$ coincide.
\end{cor}

\section{Stanley-Reisner presentations}\label{section:SRP}
In this section, we use Theorem~\ref{thm:TORIC-I} to give a description of the 
equivariant cobordism ring of a smooth toric variety which is analogous to the
stanley-Reisner presentation of the equivariant $K$-theory ({\sl cf.} 
\cite[Theorem~6.1]{VV}) and equivariant cohomology 
({\sl cf.} \cite[Theorem~8]{BDC}). This will prove Theorem~\ref{thm:COBT**}.
We follow the notations of \cite{VV} in this description.

Let $T$ be a split torus of rank $n$ and let $M$ denote the lattice of the
one-parameter subgroups of $T$. Let $X = X(\Delta)$ be a smooth toric
variety associated to a fan $\Delta$ in $M_{\R}$. For $r \ge 1$, let 
$\Delta_r$ denote the set of $r$-dimensional cones in $\Delta$.
For $\sigma \in \Delta_{\rm max}$, let $M_{\sigma}$ denote the group of 
one-parameter subgroups of $T_{\sigma}$ so that $\widehat{T}_{\sigma} =
{M}^{\vee}_{\sigma}$ as an abelian group. For any $\rho \in \Delta_1$, let
$v_{\rho}$ denote the generator of the monoid $\rho \cap M$. 
Note that if $\{\rho_1, \cdots , \rho_s\}$ is the set of one-dimensional faces
of $\sigma \in \Delta_{\rm max}$, then the smoothness of $X$ implies that
$\{v_{\rho_1}, \cdots , v_{\rho_s}\}$ is a basis of $M_{\sigma}$. Let 
$\{{v}^{\vee}_{\rho_1}, \cdots , {v}^{\vee}_{\rho_s}\}$ denote the dual basis
of $M^{\vee}_{\sigma}$. 

We also recall ({\sl cf.} ~\eqref{eqn:Chern-map}) that for 
$\sigma \in \Delta$, there is a canonical 
embedding of abelian groups $\widehat{T}_{\sigma} \inj 
\left(\Omega^1_{T_{\sigma}}(k), F\right)$ given by 
$\chi \mapsto c^{T_{\sigma}}_1\left(L_{\chi}\right)$, if we consider the 
addition in $\Omega^1_{T_{\sigma}}(k)$ according to the formal group law of $\bL$.
In what follows, $\widehat{T}_{\sigma}$ will be
thought of as a subgroup of degree one elements in $S(T_{\sigma}) = 
\Omega^*_{T_{\sigma}}(k)$ in this sense.

For each $\rho \in \Delta_1$, we define an element $u_{\rho} = 
\left(u^{\sigma}_{\rho}\right) \in 
{\underset{\tau \in \Delta_{\rm max}}\prod} S\left(T_{\sigma}\right)$ such that
\begin{equation}\label{eqn:SRP1}
u^{\sigma}_{\rho} = \left\{ \begin{array}{ll}
{{v}^{\vee}_{\rho}} & \mbox{if \ $\rho \le \sigma$} \\
0 & \mbox{otherwise .}
\end{array}
\right.
\end{equation} 

Then $u_{\rho}$ has the property that for all $\sigma_1, \sigma_2 \in 
\Delta_{\rm max}$, the restrictions of $u^{\sigma_1}_{\rho} \in 
\widehat{T}_{\sigma_1}$ and $u^{\sigma_2}_{\rho} \in \widehat{T}_{\sigma_2}$ in 
$\widehat{T}_{\sigma_1 \cap \sigma_2}$ coincide. 
It follows from 
Theorem~\ref{thm:TORIC-I}
that each $u_{\rho}$ is an element of $\Omega^*_T(X)$.

If $S$ is a subset of $\Delta_1$ which is not contained in any maximal cone of
$\Delta$, then for any given $\sigma \in \Delta_{\rm max}$, there is one
$\rho \in S$ such that $\rho \nleq \sigma$. This implies in particular
that $u^{\sigma}_{\rho} = 0$. We conclude from this that the elements
$u_{\rho}$ satisfy the relation
\begin{equation}\label{eqn:SRP2}
{\underset{\rho \in S}\prod} u_{\rho} \ = 0 \ {\rm in} \ \Omega^*_T(X)
\end{equation}
whenever $S \subseteq \Delta_1$ is such that it is not contained in any
maximal cone of $\Delta$. We shall denote the collection of all such subsets
of $\Delta_1$ by $\Delta^0_1$.

Let $\bL[[t_{\rho}]]$ denote the graded power series ring over $\bL$ in the 
variables $\{t_{\rho}| \rho \in \Delta_1\}$ and let $I_{\Delta}$ denote the graded
ideal generated by the set of monomials 
$\{{\underset{\rho \in S}\prod} t_{\rho} | S \in  \Delta^0_1\}$.
Since $\Omega^*_T(X)$ is a subring of 
${\underset{\sigma \in \Delta_{\rm max}}\prod} S\left(T_{\sigma}\right)$
and since this latter term is a product of graded power series rings over
$\bL$, it follows from ~\eqref{eqn:SRP1} and ~\eqref{eqn:SRP2} that there
is a $\bL$-algebra homomorphism
\begin{equation}\label{eqn:SRP*}
\Psi_X : \frac{\bL[[t_{\rho}]]}{I_{\Delta}} \to \Omega^*_T(X)
\end{equation}
\[
t_{\rho} \mapsto u_{\rho}.
\]

\begin{thm}\label{thm:SRP-Main}
For a smooth toric variety $X = X(\Delta)$ associated to a fan $\Delta$ in
$M_{\R}$, the homomorphism $\Psi_X$ is an isomorphism.
\end{thm}
\begin{proof}
We prove by the induction on the number of cones in $\Delta$.
Suppose $\Delta = \{\sigma\}$ is a singleton set. In that case, $\sigma$
is the only maximal cone and we have seen in the proof of
Theorem~\ref{thm:decomposition*} that $X = U_{\sigma}$ is a $T$-equivariant
vector bundle over $O_{\sigma}$ such that the inclusion 
$O_{\sigma} \stackrel{i_{\sigma}}{\inj} X$ is the zero-section embedding. 
Hence, there is an isomorphism $\Omega^*_T(X) \stackrel{\cong}
{\underset{i^*_{\sigma}}\to}
\Omega^*_T(O_{\sigma}) \cong S(T_{\sigma}) = \bL[[t_1, \cdots , t_s]]$,
where $s$ is the dimension of $\sigma$. It is also clear in this case that
the ideal $I_{\Delta}$ in ~\eqref{eqn:SRP*} is zero. Hence, we have isomorphism
\[
\bL[[t_1, \cdots , t_s]] \stackrel{\cong}{\underset{\Psi_X}\to} \Omega^*_T(X)
\stackrel{\cong}{\underset{\Phi_X}\to} S(T_{\sigma}).
\]

We consider now the general case. We assume that $|\Delta| \ge 2$ and choose
a maximal cone $\sigma$ of dimension $s \ge 1$ in $\Delta$. Let 
$X' = X'(\Delta')$ 
be the toric variety associated to the fan $\Delta' = \Delta \setminus 
\{\sigma\}$.
Note that $O_{\sigma}$ is a closed $T$-orbit in $X$ and $X'$ is the complement
of $O_{\sigma}$ in $X$.  Let $U_{\sigma} \subset X$ be the principal open set
associated to the fan consisting of all faces of $\sigma$ and let
$U'$ be the complement of $O_{\sigma}$ in $U_{\sigma}$. Then $U'$ is nothing but
the complement of the zero-section in the $T$-equivariant vector bundle
$U_{\sigma} \to O_{\sigma}$. Let $i_{\sigma} : O_{\sigma} \inj X$ and
$j_{\sigma} : X' \inj X$ denote the closed and open embeddings respectively. 
Let $S_{\sigma} = \{\rho_1, \cdots ,\rho_s\}$ be the set of
one-dimensional faces of $\sigma$ and set 
\[
x_{\sigma} = \ \stackrel{s}{\underset{j =1}\prod}
t_{\rho_j} \in \frac{\bL[[t_{\rho}]]}{I_{\Delta}} \ \ {\rm and} \ \ 
y_{\sigma} =  \ \stackrel{s}{\underset{j =1}\prod} u_{\rho_j} \in \Omega^*_T(X).
\]
Since $N_{{O_{\sigma}}/{X}} = N_{{O_{\sigma}}/{U_{\sigma}}}$ and since the latter is
of the form $\stackrel{s}{\underset{j = 1}\oplus} L_{\chi_j}$, where 
$\{\chi_1, \cdots , \chi_s\}$ is a basis of $\widehat{T}_{\sigma}$, it follows
from the proof of Lemma~\ref{lem:rigiditysuff} and the definition of the
elements $u_{\rho}$ that 
\begin{equation}\label{eqn:SRPM1}
c^T_s\left(N_{{O_{\sigma}}/{X}}\right) = y_{\sigma} \in \Omega^*_T(X).
\end{equation}
We consider the diagram
\begin{equation}\label{eqn:SRPM2}
\xymatrix{
\bL[[t_{\rho_1}, \cdots , t_{\rho_s}]] \ar[r]^{\ \ \ \cong} \ar[d]_{x_{\sigma}} &
\Omega^*_T(O_{\sigma}) \ar[d]^{{i_{\sigma}}_*} \ar[r]^{\cong} &
S(T_{\sigma}) \ar[d]^{y_{\sigma}} \\
{\frac{\bL[[t_{\rho}]]}{I_{\Delta}}} \ar[r]^{\Psi_X} &
\Omega^*_T(X) \ar[r]^<<<<{\Phi_X} & 
{{\underset{\tau \in \Delta_{\rm max}}\prod} S(T_{\tau})},}
\end{equation}
where the horizontal maps on the top are the obvious isomorphisms
taking $t_{\rho_j}$ to $u_{\rho_j}$. The left and the right vertical maps
are the multiplication by the indicated elements in the target rings.
We claim that all the vertical arrows are injective and the left square in 
this diagram commutes. 

To prove the claim, notice that the composite outer square clearly commutes
by the definition of $x_{\sigma}$ and $y_{\sigma}$ and the map $\Psi_X$.
Since $\Phi_X$ is injective by Theorem~\ref{thm:TORIC-I}, we only need to 
show that the right square commutes and the right vertical arrow is injective
to prove the claim.

We first observe that the right vertical arrow is the multiplication by
$y_{\sigma}$ on the factor $S(T_{\sigma})$ and is zero on the other factors
of ${{\underset{\tau \in \Delta_{\rm max}}\prod} S(T_{\tau})}$. Thus the
required injectivity is equivalent to showing that the multiplication by
$y_{\sigma}$ is injective in $S(T_{\sigma})$. But this is obvious since
$S(T_{\sigma}) \cong \bL[[t_{\rho_1}, \cdots , t_{\rho_s}]]$ is an integral domain
and $y_{\sigma}$ is clearly non-zero ({\sl cf.} \cite[Lemma~5.3]{Krishna3}). 

To show the commutativity of the right square, we observe from the
proof of Theorem~\ref{thm:TORIC-I} that $\Phi_X$ is simply the product of
the pull-back maps $i^*_{\tau} : \Omega^*_T(X) \to \Omega^*_T(O_{\tau})$,
$\tau \in \Delta_{\rm max}$.
Hence the composite $\Phi_X \circ {i_{\sigma}}_*$ is $i^*_{\sigma} \circ
{i_{\sigma}}_*$ on the factor $S(T_{\sigma})$ and zero on the other factors of
${{\underset{\tau \in \Delta_{\rm max}}\prod} S(T_{\tau})}$.
Since we have just
seen that the composite $S(T_{\sigma}) \xrightarrow{y_{\sigma}}
{{\underset{\tau \in \Delta_{\rm max}}\prod} S(T_{\tau})}$ is of similar type,
we are reduced to showing that the triangle
\[
\xymatrix@C.9pc{
\Omega^*_T(O_{\sigma}) \ar[dr]^{y_{\sigma}} \ar[d]_{{i_{\sigma}}_*} & \\
\Omega^*_T(X) \ar[r]_{i^*_{\sigma}} & \Omega^*_T(O_{\sigma})}
\]
commutes. But this follows immediately from Proposition~\ref{prop:SIF}
and ~\eqref{eqn:SRPM1}.

To complete the proof of the theorem, we now consider the
diagram
\begin{equation}\label{eqn:SRPM3}
\xymatrix{
0 \ar[r] & {\bL[[t_{\rho_1}, \cdots , t_{\rho_s}]]} \ar[r]^{\ \ \ \ \ x_{\sigma}} 
\ar[d]_{\cong} & {\frac{\bL[[t_{\rho}]]}{I_{\Delta}}} \ar[d]^{\Psi_X}
\ar[r]^{{\ov{j}}_{\sigma}^*} & {\frac{\bL[[t_{\rho}]]}{(I_{\Delta}, x_{\sigma})}} 
\ar[d]^{{\ov{\Psi}}_{X'}} \ar[r] & 0 \\
0 \ar[r] & {\Omega^*_T(O_{\sigma})} \ar[r]_{{i_{\sigma}}_*} &
\Omega^*_T(X) \ar[r]_{j^*_{\sigma}} & \Omega^*_T(X') \ar[r] & 0,}
\end{equation}
where ${{\ov{j}}_{\sigma}^*}$ is the natural quotient map by the ideal
$(x_{\sigma})$ in $\frac{\bL[[t_{\rho}]]}{I_{\Delta}}$. Note that the image of the
first map in the top row is the ideal $(x_{\sigma})$ because the product
of $x_{\sigma}$ with any $t_{\rho}, \rho \notin \{\rho_1, \cdots , \rho_s\}$ 
is zero. 
The left square in this diagram commutes and the first maps in both the
rows are injective by the above claim. The bottom row is exact by 
Theorem~\ref{thm:Basic}. Since $\sigma$ is not a cone of $\Delta'$,
the element $x_{\sigma}$ is zero in 
$\frac{\bL[[t_{\rho}, \ \rho \in {\Delta'}_1]]}{I_{\Delta'}}$ and hence
the map $j^*_{\sigma} \circ \Psi_X$ has a factorization:
\[
\frac{\bL[[t_{\rho}]]}{I_{\Delta}} \surj 
{\frac{\bL[[t_{\rho}]]}{(I_{\Delta}, x_{\sigma})}} 
\to \frac{\bL[[t_{\rho}, \ \rho \in {\Delta'}_1]]}{I_{\Delta'}}
\xrightarrow{\Psi_{X'}} \Omega^*_T(X'),
\]
where the middle arrow is the natural map of the Stanley-Reisner power series
rings induced by the inclusion of the fans $\Delta' \subset \Delta$. 
Letting ${{\ov{\Psi}}_{X'}}$ denote the composite 
${\frac{\bL[[t_{\rho}]]}{(I_{\Delta}, x_{\sigma})}} 
\to \frac{\bL[[t_{\rho}, \ \rho \in {\Delta'}_1]]}{I_{\Delta'}}
\xrightarrow{\Psi_{X'}} \Omega^*_T(X')$, we see that  
the right square in the above diagram also commutes.

If all the cones of $\Delta$ are at most one-dimensional, then 
$x_{\sigma} = t_{\rho}$, where $\rho = \sigma$ and it is
obvious that ${\frac{\bL[[t_{\rho}]]}{(I_{\Delta}, x_{\sigma})}}$ is the
Stanley-Reisner ring associated to the fan $\Delta'$.
If $\Delta$ has a cone of dimension at least two, we can assume that $\sigma$
is of dimension at least two. In that case, we have ${\Delta'}_1 =
\Delta_1$ and the natural inclusion $\Delta^0_1 \subseteq {\Delta'}^0_1$ gives
the equality  ${\Delta'}^0_1 = {\Delta}^0_1 \ \coprod \ \{S_{\sigma}\}$. 
In particular, we have 
\[
{\frac{\bL[[t_{\rho}, \ \rho \in \Delta_1]]}{(I_{\Delta}, x_{\sigma})}} 
\xrightarrow{\cong}
\frac{\bL[[t_{\rho}, \ \rho \in {\Delta'}_1]]}{I_{\Delta'}}.
\]
On the other hand, $\Delta'$ is a fan with smaller number of cones than in
$\Delta$ and $X' = X'(\Delta')$. Hence the map 
$\frac{\bL[[t_{\rho}, \ \rho \in {\Delta'}_1]]}{I_{\Delta'}} 
\xrightarrow{\Psi_{X'}} \Omega^*_T(X')$ is an isomorphism by induction. We 
conclude that the map ${\ov{\Psi}}_{X'}$ in the diagram~\eqref{eqn:SRPM3}
is an isomorphism.
A diagram chase shows that $\Psi_X$ is also an isomorphism.
This completes the proof of the theorem.
\end{proof}

\section{Cobordism ring of toric varieties}\label{section:CRTV}
Let $T$ be a split torus of rank $n$ with the group of one-parameter
subgroups $M$.
We now describe the ordinary cobordism ring of a smooth toric variety
using Theorem~\ref{thm:SRP-Main} and the following result which explicitly
describes the ordinary cobordism ring as a quotient of the equivariant
cobordism ring. Recall from Section~\ref{section:AC} that for any
linear algebraic group $G$ acting on a scheme $X$, there is a natural
forgetful map $r^G_X: \Omega^*_G(X) \to \Omega^*(X)$ which is a ring
homomorphism if $X$ is smooth.

\begin{thm}\label{thm:FF*}$(${\sl cf.} \cite[Theorem~3.4]{Krishna3}$)$
Let $T$ be a split torus acting on a $k$-variety $X$. Then the forgetful
map $r^T_X$ induces an isomorphism
\[
\Omega^T_*(X) \otimes_{S} \bL \xrightarrow{\cong} \Omega_*(X).
\]
If $X$ is smooth, this is an $\bL$-algebra isomorphism.
\end{thm}
Let $X = X(\Delta)$ be a smooth toric variety associated to a fan $\Delta$
in $M_{\R}$. Let $<,> : M \times M^{\vee} \to \Z$ denote the natural pairing.
Recall that for an one-dimensional cone $\rho \in \Delta$, the symbol
$v_{\rho}$ denotes the generator of the monoid $M \cap \rho$. 
It is well known that the associated orbit closure $V_{\rho} = \ov{O_{\rho}}$ is 
also a smooth toric variety which is a $T$-equivariant Weil divisor on $X$. 
In particular, $[V_{\rho} \to X]$
is a $T$-equivariant cobordism cycle on $X$. Let $[V_{\rho}] \in \Omega^1_T(X)$
denote its fundamental class ({\sl cf.} Subsection~\ref{subsection:FundC}).
We denote the cobordism cycle $[V_{\rho} \to X] \in \Omega^*(X)$ also by
$[V_{\rho}]$.

Let $\bL[t_{\rho}] = \bL[t_{\rho}, \rho \in \Delta_1]$ be the graded polynomial 
ring over $\bL$ with each $t_{\rho}$ homogeneous of degree one. Let
$\ov{I}_{\Delta}$ denote the graded ideal in $\bL[t_{\rho}]$ generated by the set 
of monomials $\{{\underset{\rho \in S}\prod} t_{\rho} | S \in  \Delta^0_1\} 
\bigcup \{t^{n+1}_{\rho} | \rho \in \Delta_1\}$. It is clear that the sum
${\underset{\rho \in \Delta_1}\sum} [<\chi, v_{\rho}>]_F \ t_{\rho}$
({\sl cf.} Subsection~\ref{subsection:FGL*})
is a well-defined homogeneous element of degree one in the graded ring
${\bL[t_{\rho}]}/{\ov{I}_{\Delta}}$ for every $\chi \in M^{\vee}$.

\begin{thm}\label{thm:COBT}
Let $X = X(\Delta)$ be a smooth toric variety associated to a fan $\Delta$
in $M_{\R}$. Then the assignment $t_{\rho} \mapsto [V_{\rho}]$ defines an
$\bL$-algebra isomorphism
\begin{equation}\label{eqn:COBT1}
\ov{\Psi}_X: \frac{\bL[t_{\rho}]}
{\left(\ov{I}_{\Delta}, \underset{\rho \in \Delta_1}\sum
[<\chi, v_{\rho}>]_F \ t_{\rho}\right)} \to \Omega^*(X),
\end{equation}
where $\chi$ runs over ${M}^{\vee}$.
\end{thm}
\begin{proof}
We first show that there is an $\bL$-algebra isomorphism
\begin{equation}\label{eqn:COBT1*}
\wt{\Psi}_X: \frac{\bL[[t_{\rho}]]}
{\left(I_{\Delta}, \underset{\rho \in \Delta_1}\sum
[<\chi, v_{\rho}>]_F \ t_{\rho}\right)} \to \Omega^*(X),
\end{equation}
where $\chi$ runs over ${M}^{\vee}$.
Since the orbit closure $V_{\rho}$ associated to an one-dimensional cone
$\rho$ is the union of the orbits $O_{\sigma}$ such that $\rho \le \sigma$,
it is clear from our definition of $u_{\rho}$ in ~\eqref{eqn:SRP1} that it is 
precisely the class of the cycle $V_{\rho}$ in 
${{\underset{\tau \in \Delta_{\rm max}}\prod} S(T_{\tau})}$ and hence in
$\Omega^*_T(X)$. Thus, the isomorphism in ~\eqref{eqn:COBT1*} follows 
immediately from Theorems~\ref{thm:SRP-Main} and ~\ref{thm:FF*}, once we know 
that the class of $c^T_1(L_{\chi})$ in 
${{\underset{\tau \in \Delta_{\rm max}}\prod} S(T_{\tau})}$ is
$\underset{\rho \in \Delta_1}\sum [<\chi, v_{\rho}>]_F \ u_{\rho}$
for a character $\chi$ of $T$.

Let $V_{\Delta} \subseteq {\underset{\sigma \in \Delta_{\rm max}}\prod}
\widehat{T}_{\sigma}$ be the subgroup consisting of elements $(x^{\sigma})$
with the property that for all $\sigma_1, \sigma_2 \in \Delta_{\rm max}$, the
restrictions of $x^{\sigma_1} \in \widehat{T}_{\sigma_1}$ and 
$x^{\sigma_2} \in \widehat{T}_{\sigma_2}$ in $\widehat{T}_{\sigma_1 \cap \sigma_2}$
coincide. It is easy to check that the elements $u_{\rho}$ form a basis of
$V_{\Delta}$ ({\sl cf.} \cite[Proposition~6.3]{VV}). Here, $\widehat{T}_{\sigma}$ 
is thought of as a subgroup of degree one elements in $S(T_{\sigma})$ as
in Section~\ref{section:SRP}. As an element of 
${\underset{\sigma \in \Delta_{\rm max}}\prod} \widehat{T}_{\sigma}$,
we have 
\[
\begin{array}{lll}
\chi & = &
{\underset{\sigma \in \Delta_{\rm max}}\sum}
{\underset{\rho \in \Delta_1}{\underset{\rho \le \sigma}\sum}}
<\chi, v_{\rho}> v^{\vee}_{\rho} \\
& = & 
{\underset{\rho \in \Delta_1}\sum}{\underset{\sigma \in \Delta_{\rm max}}
{\underset{\rho \le \sigma}\sum}} <\chi, v_{\rho}> v^{\vee}_{\rho} \\
&  = & 
{\underset{\rho \in \Delta_1}\sum} <\chi, v_{\rho}> u_{\rho}.
\end{array}
\]
In particular, we get 
$c^T_1(L_{\chi}) = \underset{\rho \in \Delta_1}\sum [<\chi, v_{\rho}>]_F \ u_{\rho}$.
This proves the isomorphism of ~\eqref{eqn:COBT1*}. 

To complete the proof of the theorem, we only have to show that the
inclusion $\bL[t_{\rho}] \subset \bL[[t_{\rho}]]$ descends to an isomorphism
of graded rings 
\begin{equation}\label{eqn:COBT2}
\frac{\bL[t_{\rho}]}{\left(\ov{I}_{\Delta}, \underset{\rho \in \Delta_1}\sum
[<\chi, v_{\rho}>]_F \ t_{\rho}\right)} \xrightarrow{\cong}
\frac{\bL[[t_{\rho}]]}
{\left(I_{\Delta}, \underset{\rho \in \Delta_1}\sum
[<\chi, v_{\rho}>]_F \ t_{\rho}\right)}.
\end{equation}

Let $A$ and $B$ denote the rings on the left and the right hand sides of
~\eqref{eqn:COBT2} respectively.
We first observe that $\wt{\Psi}_X(t_{\rho}) = [V_{\rho}]$, which is a homogeneous
element of degree one in $\Omega^*(X)$. Since $\Omega^{>n}(X) = 0$, we
conclude from ~\eqref{eqn:COBT1*} that $t^{n+1}_{\rho} = 0$ in $B$
for each $\rho \in \Delta_1$. 
Since ${\bL[t_{\rho}]}/{(t^{n+1}_{\rho})} \xrightarrow{\cong}
{\bL[[t_{\rho}]]}/{(t^{n+1}_{\rho})}$, the isomorphism of ~\eqref{eqn:COBT2} is now
immediate.  
\end{proof}

\begin{exms}\label{exm:PAT}
We now illustrate Theorem~\ref{thm:COBT} by using it to
verify the known formulae for cobordism rings of some standard toric 
varieties. If $X = X(\Delta)$, where $\Delta$ consists of all the faces 
of a single cone $\sigma$ in $M_{\R}$, then the smoothness of $X$ implies that
the set of primitive vectors $\{v_1, \cdots , v_s\}$ corresponding to the
one-dimensional faces $\{\rho_1, \cdots , \rho_s\}$ of $\sigma$ can be
extended to a basis $\{v_1, \cdots , v_s, v_{s+1}, \cdots , v_n\}$ of $M$.
Setting $\chi_i = v^{\vee}_i$, we get $I_{\Delta} = 0 = c^T_1(L_{\chi_i})$
for $s+1 \le i \le n$ and $c^T_1(L_{\chi_i}) = t_{\rho_i}$ for $1 \le i \le s$.
In particular, we obtain
\[
\Omega^*(X) \cong \frac{\bL[t_{\rho_1}, \cdots , t_{\rho_s}]}{\left(I_{\Delta}, 
\underset{\rho \in \Delta_1}\sum [<\chi_i, v_{\rho}>]_F \ t_{\rho}\right)}
= \frac{\bL[t_{\rho_1}, \cdots , t_{\rho_s}]}{\left(t_{\rho_1}, \cdots ,
t_{\rho_s}\right)} = \bL.
\]

Next we consider the case of $\P^n = X(\Delta)$. We take
$M = \Z v_1 \oplus \cdots \oplus \Z v_n$ and let 
$\{\rho_1, \cdots , \rho_{n+1}\}$ be the set of one-dimensional cones
in $\Delta$, where $\rho_i$ is the edge along the vector $v_{i}$ for
$1 \le i \le n$ and $\rho_{n+1}$ is the edge along the primitive vector 
$v_{n+1} = -(v_1 + \cdots + v_n)$. Set $\chi_i = v^{\vee}_i$ for $1 \le i \le n$.
It is then easy to see that if $c^T_i$ denotes the element
$\underset{\rho \in \Delta_1}\sum [<\chi_i, v_{\rho}>]_F \ t_{\rho}$ in
$\bL[[t_{\rho_1}, \cdots , t_{\rho_{n+1}}]]$, then
$c^T_i = t_{\rho_i} -_F t_{\rho_{n+1}}$ for $1 \le i \le n$.
In particular, we get
\[
c^T_i = t_{\rho_i} - t_{\rho_{n+1}} +  
t_{\rho_i} t_{\rho_{n+1}}f\left(t_{\rho_i}, t_{\rho_{n+1}}\right)
= t_{\rho_i} - t_{\rho_{n+1}}\left(1 + t_{\rho_i}
g\left(t_{\rho_i}, t_{\rho_{n+1}}\right)\right)
= t_{\rho_i} - u_i t_{\rho_{n+1}},
\]
where $u_i$ is an invertible homogeneous element of degree zero
in $\bL[[t_{\rho_1}, \cdots , t_{\rho_{n+1}}]]$.
Since $I_{\Delta} = \left(\stackrel{n+1}{\underset{i = 1}\prod} t_{\rho_i}\right)$,
we get
\[
\begin{array}{lll}
\Omega^*(\P^n) & \cong & 
\frac{\bL[[t_{\rho_1}, \cdots , t_{\rho_{n+1}}]]}
{\left(I_{\Delta}, \stackrel{n+1}{\underset{j =1}\sum}
[<\chi_i, v_{\rho_j}>]_F \ t_{\rho_j}\right)} \\
& \cong & 
\frac{\bL[[t_{\rho_1}, \cdots , t_{\rho_{n+1}}]]}
{\left(\stackrel{n+1}{\underset{i = 1}\prod} t_{\rho_i},
t_{\rho_1} - u_1 t_{\rho_{n+1}}, \cdots ,
t_{\rho_n} - u_n t_{\rho_{n+1}}\right)} \\
& \cong & 
\frac{\bL[[t_{\rho_{n+1}}]]}{\left(t^{n+1}_{\rho_{n+1}}\right)} \\
& \cong & \frac{\bL[t_{\rho_{n+1}}]}
{\left(t^{n+1}_{\rho_{n+1}}\right)}.
\end{array}
\]
\end{exms}

\noindent\emph{Acknowledgments.}
This paper grew out of some discussions the authors had during their
stay at the Institut Fourier, Universit\'e de 
Grenoble in July, 2010. The authors take this opportunity to thank Michel 
Brion for invitation and financial support during the visit.

\end{document}